\documentclass[reqno]{amsart}

\usepackage{amsmath}
\usepackage{amsfonts}
\usepackage{amssymb,enumerate,comment,mathdots, stmaryrd}
\usepackage{amsthm}
\usepackage[all]{xy}
\usepackage{rotating}
\usepackage{hyperref}

\usepackage{mathrsfs}

\usepackage{graphicx}
\usepackage{tcolorbox}

\theoremstyle{plain}
\newtheorem{lem}{Lemma}[section]
\newtheorem{cor}[lem]{Corollary}
\newtheorem{prop}[lem]{Proposition}
\newtheorem{thm}[lem]{Theorem}

\theoremstyle{definition}
\newtheorem{defn}[lem]{Definition}
\newtheorem{ex}[lem]{Example}
\newtheorem{question}[lem]{Question}

\newtheorem{disc}[lem]{Remark}
\newtheorem{rmk}[lem]{Remark}

\newtheorem{notn}[lem]{Notation}
\newtheorem{fact}[lem]{Fact}

\newtheorem{assumption}[lem]{Assumption}

\theoremstyle{remark}

\newcommand{\Hom}{\operatorname{Hom}}

\newcommand{\s}{\mathfrak{S}}

\newcommand{\im}{\operatorname{Im}}

\newcommand{\Cl}{\operatorname{Cl}}

\newcommand{\cone}{\operatorname{Cone}}
\newcommand{\Ker}{\operatorname{Ker}}

\newcommand{\ideal}[1]{\mathfrak{#1}}

\newcommand{\p}{\ideal{p}}
\newcommand{\q}{\ideal{q}}

\newcommand{\fa}{\ideal{a}}
\newcommand{\fb}{\ideal{b}}

\newcommand{\sfk}{\mathsf k}

\newcommand{\ol}{\overline}
\newcommand{\wti}{\widetilde}

\newcommand{\bbz}{\mathbb{Z}}
\newcommand{\bbn}{\mathbb{N}}
\newcommand{\bbq}{\mathbb{Q}}

\newcommand{\xra}{\xrightarrow}

\newcommand{\ve}{\varepsilon}

\newcommand{\te}{\theta}
\newcommand{\x}{\mathbf{x}}

\newcommand{\tri}{\trianglelefteq}

\newcommand{\gd}{\delta}
\newcommand{\gs}{\sigma}
\newcommand{\gl}{\lambda}
\newcommand{\gz}{\zeta}

\renewcommand{\geq}{\geqslant}
\renewcommand{\leq}{\leqslant}

\renewcommand{\ker}{\Ker}

\newcommand{\Ext}[4][R]{\operatorname{Ext}_{#1}^{#2}(#3,#4)}

\renewcommand{\Hom}[3][R]{\operatorname{Hom}_{#1}(#2,#3)}

\newcommand{\ssm}{\smallsetminus}

\numberwithin{equation}{lem}

\newcommand{\mfx}{\mathfrak{X}}
\newcommand{\ull}{\underline\ell}

\begin{document}

\bibliographystyle{amsplain}

\author{Sean K. Sather-Wagstaff}

\address{School of Mathematical and Statistical Sciences,
Clemson University,
O-110 Martin Hall, Box 340975, Clemson, S.C. 29634,
USA}

\email{ssather@clemson.edu}

\urladdr{https://ssather.people.clemson.edu/}

\author{Tony Se}

\address{Department of Mathematics,
West Virginia University,
320 Armstrong Hall,
P.O. Box 6310,
Morgantown, WV,
26506-6310,
USA}

\email{tony.se@mail.wvu.edu}

\urladdr{http://community.wvu.edu/~tts00001/}

\author{Sandra Spiroff}

\address{Department of Mathematics,
University of Mississippi,
Hume Hall 335, P.O. Box 1848, University, MS 38677,
USA}

\email{spiroff@olemiss.edu}

\urladdr{http://math.olemiss.edu/sandra-spiroff/}

\thanks{
Sandra Spiroff was supported in part by Simons Foundation Collaboration Grant 584932.}

\title{Ladder determinantal rings over normal domains}




\begin{abstract} 
We explicitly describe the divisor class groups and semidualizing modules for ladder determinantal rings with coefficients in an arbitrary normal domain for arbitrary ladders, not necessarily connected, and all sizes of minors. 
\end{abstract}

\maketitle


\section{Introduction} \label{sec00}

Throughout this paper, let $\sfk$ be a field, and let $t\geq 2$ be an integer.
All rings are commutative with 1.
In this section, let $A$ be a normal domain, i.e., a noetherian, integrally closed  domain.

\

This paper investigates divisor class groups and semidualizing modules for ladder determinantal rings over $A$.  Ladder determinantal rings 
generalize the more classical determinantal rings that are central to the study Grassmannian and Scubert varieties~\cite{bruns:dr}, and they are useful for investigating Young tableaux~\cite{MR926272}.
These rings feature in a number of publications, e.g., \cite{Co, MR1413891, SWSeSpP1, SWSeSpP2}, hence we only briefly recall that a {\bf ladder} is a subset $Y$ of an $m \times n$ matrix $X=(X_{ij})$ of indeterminates satisfying the property that if $X_{ij},X_{pq}\in Y$ satisfy $i\leq p$ and $j\leq q$, then $X_{iq},X_{pj}\in Y$. To avoid trivialities, we assume that $X$ is the smallest matrix containing $Y$ and that every row and column of $Y$ is non-empty.
Let $I_t(Y)$  be the ideal of the polynomial ring $A[Y]$ generated by the $t \times t$ minors lying entirely in $Y$. Then $A_t(Y) := A[Y]/I_t(Y)$ is a {\bf ladder determinantal ring} of $t$-minors.  (Note that this notation differs slightly from that in \cite{Co, MR1413891,SWSeSpP1,SWSeSpP2}.  We require the extra flexibility afforded by this notation since we analyze the ladder construction over  
different coefficient
rings.)

In the main result of this paper, Theorem~\ref{thm181230a}, we determine how many non-isomorphic semidualizing modules the ring $A_t(Y)$ has, where $A$ is not necessarily a field, and regardless of the size of the $t$-minors or any connectedness conditions on $Y$.  Recall that for a commutative noetherian ring $R$, a finitely generated $R$-module $C$ is \textbf{semidualizing} if $\Hom CC\cong R$ and $\Ext iCC=0$ for all $i\geq 1$. 
The set  of isomorphism classes of 
semidualizing
$R$-modules is denoted $\s_0(R)$.
One reason to want to understand semidualizing modules is that they provide nice dualities.
For instance, the free $R$-module of rank 1 is semidualizing, giving rise to the classical duality $(-)^*=\Hom -R$
which is crucial, e.g., for Auslander and Bridger's G-dimension~\cite{auslander:smt}.
For another example, if $R$ is Cohen-Macaulay and either 
complete
local or standard graded, then the canonical module $\omega$ of $R$ is semidualizing, yielding 
Grothendieck's~\cite{hartshorne:lc} local duality $(-)^\dagger=\Hom -\omega$. 

We now summarize the contents of the paper.
In Section~\ref{sec190719a} we prove our main result, Theorem~\ref{thm181230a}.  
Section~\ref{sec190719b} analyzes divisor class groups of ladder determinantal rings over normal domains, the key tool for our proof of Theorem~\ref{thm181230a}; see Corollary~\ref{cor190402a} and Theorem~\ref{thm190402a} for our main conclusions about 
divisor class groups.
Appendix~\ref{sec181225bx} contains technical results about ladders and ladder determinantal rings, including two isomorphisms for localizations for use in the other sections.


We conclude this Introduction with a few facts for use throughout the paper.

\begin{disc}\label{disc181224a}
If $R$ is an integral domain, then $R$ either contains a field or an isomorphic copy of $\bbz$.
In particular, $R$ contains a subring $D$ that is  a principal ideal domain. (In this paper, we consider fields to be principal ideal domains.)
With this set-up, every torsion-free $R$-module is torsion-free over $D$, hence flat over $D$; in particular, $R$ is flat over $D$, as is every reflexive $R$-module.
\end{disc}

In this paper, we use the description of the {\bf divisor class group} of 
our
normal domain $A$ as the set of isomorphism classes
of finitely generated rank-1 reflexive $A$-modules with operations
$[\fa]+[\fb]=[(\fa\otimes_A\fb)^{**}]$ and $[\fa]-[\fb]=[\Hom[A]\fb\fa]$, and with additive identity $[A]$.

\begin{fact}[\protect{\cite[Proposition~3.4]{sather:divisor}}]\label{lem190103a}
If $N$ is a semidualizing $A$-module, then $N$ is reflexive of rank 1.
In particular, we have $\s_0(A)\subseteq\Cl(A)$.
\end{fact}

\begin{ex}\label{ex181224c}
Let $Y$ be a ladder of variables,
and let
$D$ be a 
principal ideal domain contained in $A$ as a subring;
see Remark~\ref{disc181224a}.
(This construction holds for any integral domain, but our application of the example will be when the ring is normal.)
Then the natural inclusions 
make the following diagram commute
and we have
$A\otimes_DD_t(Y)\cong A_t(Y)$.
\begin{equation}\begin{split}\label{diag181225a}
\xymatrix{D\ar[r]\ar[d]&A\ar[d]\\
D_t(Y)\ar[r]&A_t(Y)
}\end{split}\end{equation}
Thus, \cite[Proposition~3.5]{SWSeSpG} 
yields a well-defined, relation-respecting 
map
$\s_0(A)\times\s_0(D_t(Y))\to\s_0(A_t(Y))$ given by 
$([C_1],[C_2])\mapsto [C_1\otimes_D C_2]$, 
where the relation $\tri$ on $\s_0(A)$ is defined via
total reflexivity, as described in~\cite{frankild:rbsc} 
and~\cite[Definition~1.5]{SWSeSpG},
and similarly for $\s_0(D_t(Y))$ and $\s_0(A_t(Y))$; and $\s_0(A)\times\s_0(D_t(Y))$ uses the product relation.
Note that $\tri$ is reflexive; it is antisymmetric if and only if the Picard group of $A$ is trivial,
and it is suspected to be transitive.
\end{ex}

\section{Divisor Class Groups}
\label{sec190719b}

\begin{assumption}\label{ass190807a}
Throughout this section, 
let $A$ be a normal domain with field of fractions $K$, and let $Y$ be a ladder of variables (not necessarily path-connected or $t$-connected).
Let $D$ be a subring of $A$ that is 
a principal ideal domain, and let $L$ denote the field of fractions of $D$.
Let $f$ be as in Notation~\ref{notn190329a}.
Recall that $A$ is flat over $D$ by Remark~\ref{disc181224a}.
\end{assumption}

Our results below use the following explicit description of $\Cl(\sfk_t(Y))$, where $Y$ is a $t$-connected ladder of variables. 
Conca~\cite[pp.~467--468]{MR1413891} states this partially, but we require slightly more detail for our applications.

\begin{fact}\label{fact190701a}
Assume that $Y$ is $t$-connected.
Conca~\cite[pp.~467--468]{MR1413891} states that $\Cl(\sfk_t(Y))$ is a free abelian group of rank $h+k^*+1=h^*+k+1$.
Moreover, using the sketch provided in \emph{loc.\ cit.} with Notation~\ref{notn190701a}, 
one finds that a basis for $\Cl(\sfk_t(Y))$ is given by the classes of the ideals
$\q_i$ where $i=1,\ldots,h+1$, and the classes of the ideals $\p_j$  where $j$ ranges through the $T_j'$ of type 1.
In particular, the basis for $\Cl(\sfk_t(Y))$ is independent of the field $\sfk$, depending only on the shape of $Y$.
From this, it follows that if $\sfk\to K$ is a field extension, then the induced map $\Cl(\sfk_t(Y))\to\Cl(K_t(Y))$ is an isomorphism.
A key ingredient of the proof is to show that the minimal primes of the element $f\in\sfk_t(Y)$ from Notation~\ref{notn190329a}
are exactly the $\q_i$, the $\q_i'$, and the $\p_j$
(for $\p_j \neq 0$).
See Remark~\ref{disc190808b} for the $t$-disconnected case.
\end{fact}

\begin{ex}\label{ex190707a}
We compute $\Cl(\sfk_3(L_i))$ for the ladders $L_i$ from Example~\ref{cornertypes}.  All of the ideals listed appear in the corresponding rows in Tables~\ref{tableq} and \ref{tablep} in Remark~\ref{disc190808a}.
\smallskip

In $L_1$, the upper inside corner $T_1'$ has type 1, so $k^*=1$. 
Since $h=1$, $\Cl(\sfk_3(L_1))\cong \bbz^3$ with basis represented by the ideals $\q_1, \q_2, \p_1$.
\smallskip

In $L_2$, the upper inside corner $T_1'$ has type 2, so $k^*=0$.
Since $h=1$, $\Cl(\sfk_3(L_2))\cong \bbz^2$ with basis represented by the ideals $\q_1, \q_2$.
\smallskip

Similarly, we have $\Cl(\sfk_3(L_3))\cong \bbz^3$ with basis represented by the ideals $\q_1, \q_2, \p_1$
and $\Cl(\sfk_3(L_4))\cong \bbz^2$ with basis represented by the ideals $\q_1, \q_2$.
\end{ex}
\medskip

Our next result confirms a statement of Conca~\cite[p.~457]{MR1413891} about splitting divisor class groups.  Note that our results do not assume Conca's Assumption~(d).

\begin{prop}
\label{prop190402a}
Assume that $Y$ is $t$-connected, and
let $A\xra{g_1} A_t(Y)\xra{h_2}K_t(Y)$ be the natural flat maps.
\begin{enumerate}[\rm(a)]
\item\label{prop190402a1}
The following sequence is split-exact:
\begin{equation}\label{eq190402a0}
0\to\Cl(A)\xra{\Cl(g_1)}
\Cl(A_t(Y))\xra{\Cl(h_2)}\Cl(K_t(Y))\to 0.
\end{equation}
In particular, $\Cl(A_t(Y))\cong\Cl(A)\times\Cl(K_t(Y))$.
\item\label{prop190402a2}
The ring $A$ is a unique factorization domain if and only if the natural map $\Cl(A_t(Y))\xra{\Cl(h_2)}\Cl(K_t(Y))$ is an isomorphism.
\end{enumerate}
\end{prop}

\begin{proof}
\eqref{prop190402a1}
Since $K_t(Y)$ is obtained from $A_t(Y)$ by inverting the non-zero elements of $A$, 
Nagata's Theorem~\cite[Corollary~7.2]{fossum:dcgkd}  
tells us that $\Cl(h_2)$ is surjective with kernel generated by all the height-1 primes of $A_t(Y)$ containing non-zero elements of $A$.
It is straightforward to show that these primes are exactly the ideals of $A_t(Y)$ extended from height-1 primes of 
$A$.
Thus, $\ker(\Cl(h_2))=\im(\Cl(g_1))$.
Furthermore, $\Cl(g_1)$ is injective since $A_t(Y)$ is faithfully flat over $A$.
This establishes the exactness of the sequence~\eqref{eq190402a0}.
The sequence splits because $\Cl(K_t(Y))$ is free by Fact~\ref{fact190701a}.

\eqref{prop190402a2}
This follows directly from part~\eqref{prop190402a1} as $A$ is a UFD if and only if $\Cl(A)=0$.
\end{proof}

We next boot-strap our way to a version of Proposition~\ref{prop190402a} with no $t$-connected assumption.
Note that part~\eqref{cor190402a1} is an improvement of part of Fact~\ref{fact190701a};
see also Proposition~\ref{prop190402b}.

\begin{cor}
\label{cor190402a}
Let $Y_1,\ldots,Y_s$ be the $t$-components of $Y$.
\begin{enumerate}[\rm(a)]
\item\label{cor190402a1}
There is an isomorphism $\Cl(K_t(Y))\cong\Cl(K_t(Y_1))\times\cdots\times \Cl(K_t(Y_s)).$
In particular, $\!\Cl(K_t(Y))\!$ is free of finite rank, and the rank is independent of the field.
\item\label{cor190402a2}
The following sequence is split-exact:
\begin{equation}\label{eq190402a1}
0\to\Cl(A)\xra{\Cl(g_1)}
\Cl(A_t(Y))\xra{\Cl(h_2)}\Cl(K_t(Y))\to 0.
\end{equation}
In particular, we have $\Cl(A_t(Y))\cong\Cl(A)\times\Cl(K_t(Y))$.
\item\label{cor190402a3}
The ring $A$ is a unique factorization domain if and only if the natural map $\Cl(A_t(Y))\xra{\Cl(h_2)}\Cl(K_t(Y))$ is an isomorphism.
\end{enumerate}
\end{cor}

\begin{proof}
\eqref{cor190402a1}
For the direct product decomposition,
induct on $s$. The base case is trivial. In the inductive case, set $Y'=Y\ssm Y_1$, which is the disjoint union $Y_2\cup\cdots\cup Y_s$.
Then $K_t(Y)\cong K_t(Y_1)\otimes_KK_t(Y')\cong(K_t(Y'))_t(Y_1)$. Apply Proposition~\ref{prop190402a}\eqref{prop190402a1}
with the normal domain $A=K_t(Y')$ and the $t$-connected ladder $Y_1$ to conclude that
$\Cl(K_t(Y))\cong\Cl(K_t(Y_1))\times\Cl(K_t(Y'))$.
Now apply the induction hypothesis to $Y'$ to obtain the desired decomposition.

Since each abelian group $\Cl(K_t(Y_i))$ is free of finite rank by Fact~\ref{fact190701a},
it  follows that $\Cl(K_t(Y))\cong\Cl(K_t(Y_1))\times\cdots\times \Cl(K_t(Y_s))$ is free of finite rank as well.
Fact~\ref{fact190701a} also establishes that the rank of each group $\Cl(K_t(Y_i))$ is independent of the field $K$, hence so is the rank of $\Cl(K_t(Y))$.

\eqref{cor190402a2}--\eqref{cor190402a3}
These are proved like Proposition~\ref{prop190402a} using part~\eqref{cor190402a1}.
\end{proof}

We improve part of Fact~\ref{fact190701a} in the next Proposition.

\begin{prop}
\label{prop190402b}
Consider the field extension $L\to K$ from~\ref{ass190807a}.
The natural map $\Cl(L_t(Y)) \to\Cl(K_t(Y))$ is an isomorphism.
\end{prop}

\begin{proof}
Let $Y_1,\ldots,Y_s$ be the $t$-components of $Y$.
Fact~\ref{fact190701a} implies that each natural map $\Cl(L_t(Y_i))\to\Cl(K_t(Y_i))$ is an isomorphism,
hence, the product map $\prod_{i=1}^s\Cl(L_t(Y_i))\to\prod_{i=1}^s\Cl(K_t(Y_i))$ is as well.
One checks that the isomorphisms from Corollary~\ref{cor190402a}\eqref{cor190402a1} make the following diagram commute
$$\xymatrix{
\Cl(L_t(Y))\ar[r]^-\cong\ar[d]
&\prod_{i=1}^s\Cl(L_t(Y_i))\ar[d]^\cong\\
\Cl(K_t(Y))\ar[r]^-\cong
&\prod_{i=1}^s\Cl(K_t(Y_i)),}$$
and as a result, the unlabeled map is an isomorphism as well.
\end{proof}

The next result is an improved version of Corollary~\ref{cor190402a}\eqref{cor190402a2}.
It contains a bit more information for use in finding semidualizing modules.

\begin{thm}
\label{thm190402a}
The following natural maps 
$$\Cl(A)\times\Cl(D_t(Y))\to\Cl(A_t(Y))\to\Cl(A_t(Y)_f)\times\Cl(K_t(Y)).$$ 
are bijections, hence
$$\Cl(A_t(Y))\cong\Cl(A_t(Y)_f)\times\Cl(K_t(Y))\cong\Cl(A)\times\Cl(K_t(Y)).$$
\end{thm}

\begin{proof}
Consider the following natural maps.
\begin{gather*}
A\xra{g_1} A_t(Y)\xra{h_1}A_t(Y)_f\\
D_t(Y)\xra{g_2} A_t(Y)\xra{h_2}K_t(Y)
\end{gather*}
The map $g_1$ is flat (in fact, free and faithfully flat),
and $g_2$ is flat because $A$ is flat over $D$. 
Also, each $h_i$ is a localization map, hence flat. 

Next, there are commutative diagrams of flat maps of normal domains
$$\xymatrix{
A\ar[r]^-{g_1}\ar[d]
&A_t(Y)\ar[d]^{h_2}
&D_t(Y)\ar[r]^-{g_2}\ar[d]
&A_t(Y)\ar[d]^{h_1}
\\
K\ar[r]
&K_t(Y)
&D_t(Y)_f\ar[r]
&A_t(Y)_f
}$$
which induce commutative diagrams on divisor class groups
\begin{equation}\label{diag190402a}
\begin{split}
\xymatrix{
\Cl(A)\ar[r]^-{\Cl(g_1)}\ar[d]
&\Cl(A_t(Y))\ar[d]^{\Cl(h_2)}
&\Cl(D_t(Y))\ar[r]^-{\Cl(g_2)}\ar[d]
&\Cl(A_t(Y))\ar[d]^{\Cl(h_1)}
\\
0=\Cl(K)\ar[r]
&\Cl(K_t(Y))
&0=\Cl(D_t(Y)_f)\ar[r]
&\Cl(A_t(Y)_f).
}
\end{split}
\end{equation}
The first vanishing in~\eqref{diag190402a} comes from the fact that $K$ is a field. 
For the second vanishing, recall that $D$ is a 
principal ideal domain.  Since it is a unique factorization domain, so is 
$D[B]_F\cong D_t(Y)_f$, where $B$ and the isomorphism
are from Lemma~\ref{lem190329bxxx}.
These vanishings imply that, whenever $i\neq j$,
\begin{equation}\label{eq190402b}
0=\Cl(h_i)\circ\Cl(g_j)=\Cl(h_i\circ g_j).
\end{equation}

Next, consider the following diagram of group homomorphisms.
\begin{equation}
\label{diag190402b}
\begin{split}
\xymatrix{
\Cl(A)\times\Cl(D_t(Y))
\ar[rr]^-{\Cl(h_1\circ g_1)\times\Cl(h_2\circ g_2)}
\ar[rd]_{\Cl(g_1)+\Cl(g_2)\ \ }
&&
\Cl(A_t(Y)_f)\times\Cl(K_t(Y))\\
&\Cl(A_t(Y))
\ar[ru]_{\ \ (\Cl(h_1),\Cl(h_2))}
}
\end{split}
\end{equation}
A straightforward diagram chase using~\eqref{eq190402b} shows that this diagram commutes.

We claim that the horizontal map in diagram~\eqref{diag190402b} is an isomorphism. 
To show this, it suffices to show that the component maps 
$\Cl(h_1\circ g_1)$ and $\Cl(h_2\circ g_2)$
are isomorphisms. 
The first of these comes from the following natural commutative diagrams,
wherein the vertical isomorphism comes from the fact that the element $F$ from Notation~\ref{notn190329a} is a product of prime elements of the polynomial ring $A[B]$
from Lemma~\ref{lem190329bxxx}.
\begin{equation}
\label{diag190402c}
\begin{split}
\xymatrix{
A\ar[r]^-{h_1\circ g_1}\ar[d]
&A_t(Y)_f
&\Cl(A)\ar[r]^-{\Cl(h_1\circ g_1)}\ar[d]_\cong
&\Cl(A_t(Y)_f)\\
A[B]_F\ar[ru]_-\cong^-{\ref{lem190329bxxx}}
&&\Cl(A[B]_F)\ar[ru]_-\cong
}
\end{split}
\end{equation}
For the map $\Cl(h_2\circ g_2)\colon\Cl(D_t(Y))\to\Cl(K_t(Y))$, recall that $L$ is a subfield of $K$.
Factor the map $h_2\circ g_2\colon D_t(Y)\to K_t(Y)$ as the composition of the natural flat maps
$D_t(Y)\to L_t(Y)\to K_t(Y)$. It suffices to show that each map
$\Cl(D_t(Y))\to \Cl(L_t(Y))\to \Cl(K_t(Y))$ is an isomorphism, but this follows from 
Corollary~\ref{cor190402a}\eqref{cor190402a3} and Proposition~\ref{prop190402b}, respectively.
This establishes the claim.

Since the horizontal map in diagram~\eqref{diag190402b} is an isomorphism, commutativity of the diagram implies in particular
that the map $(\Cl(h_1),\Cl(h_2))$ is surjective. Thus, to complete the proof, it suffices to show that
this map is also injective. So, let $[\fa]\in\Cl(A_t(Y))$ be in $\ker(\Cl(h_1),\Cl(h_2))=\ker(\Cl(h_1))\cap\ker(\Cl(h_2))$.
Corollary~\ref{cor190402a}\eqref{cor190402a2} implies that
$\ker(\Cl(h_2))=\im(\Cl(g_1))$.
As a result, $[\fa]=\Cl(g_1)([\fb])$ for some $[\fb]\in\Cl(A)$.
Then the condition $\Cl(g_1)([\fb])=[\fa]\in\ker(\Cl(h_1))$ implies that
$$0=\Cl(h_1)(\Cl(g_1)([\fb]))=\Cl(h_1\circ g_1)([\fb]).
$$
Since it was established that $\Cl(h_1\circ g_1)$ is an isomorphism, $[\fb]=0$, and thus,
$[\fa]=\Cl(g_1)([\fb])=0$, as desired.
\end{proof}

\begin{disc}\label{disc190808b}
Let $Y_1,\ldots,Y_s$ be the $t$-components of $Y$.
For $\ell=1,\ldots,s$, let $k_\ell^*$ be the number of upper inside corners of $Y_\ell$ of type 1 (see Definition~\ref{defn:types}),
and let $h_\ell$ be the number of lower inside corners of $Y_\ell$.
The above results show that $\Cl(\sfk_t(Y))$ has a basis given by
$[(\q_i(Y_\ell))]$, with $i=1,\ldots,h_\ell+1$, and those
$[(\p_{j}(Y_\ell))]$ such that $T_{\ell,j}'$ is an upper inside corner of $Y_\ell$ of type 1.

The above paragraph shows how to
calculate $\Cl(k_t(Y)) \cong \Cl(k_{t-1}(Z))$, which is used in Section \ref{sec190719a}, e.g., Corollary \ref{thm:classgroup}, and more. 
In particular, if $Y$ is $t$-connected and $Z$ is obtained from $Y$ as in Assumption~\ref{ass190715a},
then a basis for $\Cl(\sfk_{t-1}(Z))$ 
is given as follows: let $Z_1,\ldots,Z_s$ be the $t$-components of $Z$ with respect to $(t-1)$; i.e., each $Z_\ell$ is $(t-1)$-connected.
Then a basis for $\Cl(\sfk_{t-1}(Z))$ 
is given by
$[(\q_i(Z_\ell))]$, with $i$ ranging through the lower outside corners of $Z_\ell$, or equivalently, ranging  through the lower outside corners of $Y$, and
$[(\p_{j}(Z_\ell))]$ such that $T_{\ell,j}'$ is an upper inside corner of $Z_\ell$ of type 1, or equivalently, ranging  through the upper inside corners of $Y$ of type 1.  See Remark~\ref{disc190808a} in regards to the running example.
\end{disc}


\section{Semidualizing Modules}\label{sec190719a}


\begin{assumption}\label{ass190807c}
Throughout this section, 
let $A$ be a normal domain with field of fractions $K$, and let $Y$ be a ladder of variables (not necessarily path-connected or $t$-connected).
Let $D$ be a subring of $A$ that is 
a principal ideal domain, and let $L$ denote the field of fractions of $D$.
Let $f = f_1\cdots f_{h+1}$ as in Notation~\ref{notn190329a}.
Let $B_1$ be the set of points of $Y$ of the lower border with thickness 1. Set
$Z=Y \setminus B_1$
and $x=x_{S_1}x_{S_2}\cdots x_{S_{h+1}}\in \sfk_t(Y)$
and $\mfx=X_{S_1}X_{S_2}\cdots X_{S_{h+1}}\in \sfk_t[Y]$.
\end{assumption}

\begin{lem}\label{lem190807b}
The natural map $\phi\colon\s_0(D_t(Y))\to\s_0(L_t(Y))$ is injective. 
\end{lem}

\begin{proof}
In the following commutative diagram
$$\xymatrix{
\Cl(D_t(Y))\ar[r]^-\cong
&\Cl(L_t(Y))\\
\s_0(D_t(Y))\ar[r]^-\phi\ar@{^(->}[u]
&\s_0(L_t(Y))\ar@{^(->}[u]
}$$
the top isomorphism is from Corollary~\ref{cor190402a}\eqref{cor190402a3}.
A diagram chase shows that $\phi$ is injective.
\end{proof}

The theorem below describes the semidualizing modules over $A_t(Y)$ where $Y$ is arbitrary and $A$ is a normal domain. Through a series of results, we are able to improve upon this by not only providing a more detailed description of $\s_0(A_t(Y))$, but also by removing the conditions \eqref{thm181226a1}--\eqref{thm181226a2}; see Theorem~\ref{thm181230a} below.  Here and elsewhere, if $U$ and $V$ are ordered sets, then the notation $U\approx V$ means that there is a perfectly relation-respecting bijection from $U$ to $V$.

\begin{thm}
\label{thm181226a}
Assume that at least one of the following conditions is satisfied:
\begin{enumerate}[\quad\rm(1)]
\item\label{thm181226a1}
$A$ contains a field, or
\item\label{thm181226a2}
the natural map $\s_0(D_t(Y))\to\s_0(L_t(Y))$ is surjective.
\end{enumerate}
Then the natural maps 
$$\s_0(A)\times\s_0(D_t(Y))\to\s_0(A_t(Y))\to\s_0(A_t(Y)_f)\times\s_0(K_t(Y))$$ 
are bijections, hence
\begin{align*}
\s_0(A_t(Y))&\approx\s_0(A_t(Y)_f)\times\s_0(K_t(Y))\approx\s_0(A)\times\s_0(K_t(Y)).
\end{align*}
\end{thm}

\begin{proof}
Consider the following  diagram where the vertical maps are the natural inclusions
and the upper triagram commutes as per the proof of Theorem~\ref{thm190402a}.
\begin{equation}
\label{diag181226a}
\begin{split}
\xymatrix{
\Cl(A)\times\Cl(D_t(Y))
\ar[rr]^-{\Cl(h_1\circ g_1)\times\Cl(h_2\circ g_2)}_-\cong
\ar[rd]_>>>>>>>{\Cl(g_1)+\Cl(g_2)\ \ }^-\cong
&&
\Cl(A_t(Y)_f)\times\Cl(K_t(Y))\\
&\Cl(A_t(Y))
\ar[ru]_<<<<<<<{\ \ (\Cl(h_1),\Cl(h_2))}^-\cong
\\
\s_0(A)\times\s_0(D_t(Y))\ar@{^(->}[uu]
\ar '[r][rr]^<{\s_0(h_1\circ g_1)\times\s_0(h_2\circ g_2)}
\ar[rd]_-{\s_0(g_1)\otimes\s_0(g_2)\ \ }
&&
\s_0(A_t(Y)_f)\times\s_0(K_t(Y))\ar@{^(->}[uu]\\
&\s_0(A_t(Y))\ar@{^(->}[uu]
\ar[ru]_-{\ \ \ (\s_0(h_1),\s_0(h_2))}
}
\end{split}
\end{equation}
Furthermore, we established that the three maps in the upper triagram are isomorphisms. 
The map $\s_0(g_1)\otimes\s_0(g_2)$ is defined as $([C_1],[C_2])\mapsto[C_1\otimes_DC_2]$;
the well definedness of this map is given in Example~\ref{ex181224c}.
The three quadrilateral faces of this diagram commute by~\cite[Theorem~4.4]{SWSeSpG}.
Since the vertical maps are one-to-one, a routine diagram chase shows that the lower triagram also commutes
and that each map in the lower triagram is one-to-one, hence so is the map $\s_0(h_2\circ g_2)$.
Also, the map $\s_0(h_1\circ g_1)\colon \s_0(A)\to\s_0(A[B]_F)\approx \s_0(A_t(Y)_f)$
is a bijection by~\cite[Corollary~3.11(b)]{sather:divisor}, via a diagram as in~\eqref{diag190402c}.

We claim that each of the conditions \eqref{thm181226a1} and \eqref{thm181226a2} implies that  $\s_0(h_2\circ g_2)$ is bijective.

\eqref{thm181226a1} Assume that $A$ contains a field, so $D$ contains a field $\sfk$. Then the map $\sfk_t(Y)\to K_t(Y)$ is faithfully flat
and induces an isomorphism on divisor class groups, hence, the induced map $\s_0(\sfk_t(Y))\to \s_0(K_t(Y))$
is a bijection by~\cite[Corollary~3.11(b)]{sather:divisor}. This map factors through the map $\s_0(h_2\circ g_2)\colon \s_0(D_t(Y))\to \s_0(K_t(Y))$,
therefore $\s_0(h_2\circ g_2)$ is surjective; since injectivity was already established, the argument in this case is complete.

\eqref{thm181226a2} Assume now that the natural map $\s_0(D_t(Y))\to\s_0(L_t(Y))$ is surjective.
Since this map is also injective, it is bijective.
As already observed, the map $\s_0(L_t(Y))\to\s_0(K_t(Y))$ is bijective,
hence so is the composition $\s_0(D_t(Y))\to\s_0(K_t(Y))$.

Now we complete the proof using the bottom triagram of~\eqref{diag181226a}.  Each map in this triagram is one-to-one. The claim and the paragraph preceding it show that the 
horizontal map in the triagram is a bijection. It follows that the diagonal map $(\s_0(h_1),\s_0(h_2))$ is also onto, hence a bijection. Thus, all the maps in this triagram are bijections, as desired.
\end{proof}

\begin{cor}\label{cor190808a}
Let $Y_1,\ldots,Y_s$ be the $t$-components of $Y$. 
The following natural map is bijective
$$\s_0(\sfk_t(Y_1))\times\cdots\times\s_0(\sfk_t(Y_s))\xra\iota\s_0(\sfk_t(Y_1)\otimes_\sfk\cdots\otimes_\sfk\sfk_t(Y_s))\approx\s_0(\sfk_t(Y)).$$
\end{cor}

\begin{proof}
Example~\ref{ex181224c} shows that $\iota$ is well-defined and injective, via a standard induction argument. 
For surjectivity, we argue by induction on $s\geq 1$, with the base case $s=1$ being trivial. 
Inductively, assume $s>1$ and set $\wti Y=Y\ssm Y_s$. 
Then with $A=\sfk_t(\wti Y)$ and $D=\sfk=L$, one has 
$$A_t(Y_s)\cong A\otimes_\sfk\sfk_t(Y_s)\cong \sfk_t(\wti Y)\otimes_\sfk\sfk_t(Y_s)\cong\sfk_t(Y).$$
Thus, Theorem~\ref{thm181226a}\eqref{thm181226a1} implies that the  map $\iota'$ in the next commutative diagram is a bijection,
while $\wti\iota$ is a bijection by our induction assumption.
$$\xymatrix@C=15mm{
\s_0(\sfk_t(Y_1))\times\cdots\times\s_0(\sfk_t(Y_{s-1}))\times\s_0(\sfk_t(Y_s))\ar[r]^-{\wti\iota\times\s_0(\sfk_t(Y_s))}_-\approx\ar[rd]_-\iota
&\s_0(\sfk_t(\wti Y))\times\s_0(\sfk_t(Y_s))\ar[d]^{\iota'}_\approx\\
&\s_0(\sfk_t(Y))}$$
It follows that $\iota$ is a bijection, as desired.
\end{proof}

\begin{lem} \label{lem:xisprime}
Let $B$ be an integral domain.
If $t>2$, then the element $x_{S_i}$ in $B_t(Y)$ is prime for all $i=1,\dots,h+1$.
\end{lem}

\begin{proof}
We argue by cases.

Case 1: $B=\sfk$ and the ladder $Y$ is $t$-connected. In this case,
we argue as in the proof in \cite[Proposition 4.8]{Co}. By Fact~\ref{fact190701a}, the minimal prime ideals of $f$ are $\q_1,\dots,\q_{h+1},\q'_1,\dots,
  \q'_{h+1}$, and $\p_j$ when $T'_j$ has type 1. The ring $\sfk_t(Y)$ is a Cohen-Macaulay domain
  and $x_{S_{\ell}} \notin \q_i,\q'_i,\p_j$ for all $1\leq \ell \leq h+1$,
  $1\leq i\leq h+1$ and $1\leq j\leq k$. (Recall that $x_{S_{\ell}}$ is a single variable
  and the prime ideals are defined by minors of size $t-1 > 1$.)  Thus, $f,x$ is a regular sequence in
  $\sfk_t(Y)$. By Lemma~\ref{lem190329b}, the element $x_{S_i}$ is prime in
  $\sfk_t(Y)_f$. Hence $x_{S_i}$ is prime in $\sfk_t(Y)$.

Case 2: $Y$ is $t$-connected.
Let $G$ be the field of fractions of $B$. By flatness, the quotient $B_t(Y)/(x_{S_i})$ is a subring of $G_t(Y)/(x_{S_i})$ which is an integral domain by Case 1.
Hence, the subring $B_t(Y)/(x_{S_i})$ is an integral domain as well. Hence $x_{S_i}$ is prime in $B_t(Y)$ in this case.

Case 3: $Y$ is $t$-disconnected. In this case, let $Y_1,\ldots,Y_s$ be the $t$-components of $Y$.
Our assumption implies that $s>1$. 
Re-order the $Y_j$'s if necessary to assume that   $x_{S_i}\in Y_s$.
Set $\wti Y=Y\ssm Y_s$ and $\wti B=B_t(\wti Y)$.
We have $B_t(Y)\cong \wti B_t(Y_s)$, and hence $B_t(Y)/(x_{S_i})\cong \wti B_t(Y_s)/(x_{S_i})$, which is a domain by Case 2.
\end{proof}

We can now use Proposition~\ref{prop:invertx} to give a new proof of
\cite[Theorem~5.1]{MR1413891}, where we call attention to the corrected
values of $\gl_1$ and $\gl_{h+1}$ in \cite[Theorem 5.1]{MR1413891}.  (See \cite[Example 7.3]{MR1413891}, specifically, $\sfk_3(Y_1)$, which is Gorenstein.)

\begin{cor} \label{thm:classgroup}
Assume that $Y$ is $t$-connected and satisfies Assumption~(d).  Let $[\omega]$ be the canonical class of $\sfk_t(Y)$, and let
  \[
    [\omega] = \sum_{i=1}^{h+1} \lambda_i [\q_i] + \sum_{j=1}^k \delta_j [\p_j]
  \]
  be the unique representation of $[\omega]$ with respect to the basis of
  $\Cl(\sfk_t(Y))$. Let $i_j=\min\{i\colon 1 \leq i \leq h+1, a_i+t-2>c_j\}$. If $h>0$, then:
  \begin{alignat*}{2}
    \gl_1 &= (a_1+b_1)-(a_0+b_0)+(t-2),\\
    \gl_i &= (a_i+b_i)-(a_{i-1}+b_{i-1})
    &\quad &\text{for } 1 < i \leq h,\\
    \gl_{h+1} &= (a_{h+1}+b_{h+1})-(a_h+b_h)-(t-2) && \text{and}\\
    \delta_j &= (a_{i_j}+b_{i_j}+2(t-2))-(c_j+d_j)
    && \text{for } 1 \leq j \leq k.
  \end{alignat*}
  If $h=0$, then $\gl_1=(a_1+b_1)-(a_0+b_0)$.
\end{cor}

\begin{proof}
  We argue as in the proof of~\cite[Theorem 4.9 (b)]{Co} by induction on $t$.
  The case $t=2$ is given by \cite[Corollary 2.3 and Proposition 2.4]{Co}.
  For $t>2$, the ideal $(x_{S_i})$ is prime for all $i=1,\dots,h+1$ by Lemma~\ref{lem:xisprime}; recall that $x = x_{S_1}x_{S_2}\cdots x_{S_{h+1}}$.  
  In the following commutative diagram, the first and third isomorphisms are from Nagata's Theorem~\cite[Corollary~7.2]{fossum:dcgkd},
and the second isomorphism is from
Proposition~\ref{prop:invertx}.
$$\xymatrix{
\Cl(\sfk_t(Y))\ar[r]^-\cong
&\Cl(\sfk_t(Y)_x)\ar[r]^-\cong_-{\ol\psi}
&\Cl(\sfk_{t-1}(Z)[B_1]_{\mfx})
&\Cl(\sfk_{t-1}(Z))\ar[l]_-\cong
}$$
The third isomorphism is induced by the natural flat homomorphism $\sfk_{t-1}(Z)\to\sfk_{t-1}(Z)[B_1]_{\mfx}$, thus, it respects canonical classes;
and similarly for the other two isomorphisms.
  Corollary~\ref{cor:correspondence} 
  implies that we have $\ol{\psi}([(\q_i(Y)]))=[(\q_i(Z))]$ and $\ol{\psi}([(\p_j(Y)]))=[(\p_j(Z))]$, hence,
one can read the coefficients for $[\omega]$ from the corresponding coefficients for $[\omega_{\sfk_{t-1}(Z)}]$;
as we note in Remark~\ref{disc190808a}, 
since $Y$ satisfies Assumption~(d), the ladder $Z$ is $(t-1)$-connected and satisfies Assumption~(d) with respect to $t-1$,
therefore, the natural basis of $\Cl(\sfk_{t-1}(Z))$ corresponds exactly to the natural basis of $\Cl(\sfk_t(Y))$.

Note that if $Y$ has
  corners $(a_i,b_i)$ with $0 \leq i \leq h+1$, then $Z$ has corners $(a_0+1,b_0)$,
  $(a_{h+1},b_{h+1}+1)$, and $(a_i+1,b_i+1)$ for $1 \leq i \leq h$
  (without relabeling the variables in $Z$). 
  From this, it is straightforward to check that the $\gl$'s and $\gd$'s coming from $Z$ are the same as the ones coming from $Y$.
Thus, the desired result follows by induction.
\end{proof}

Next, we describe all the semidualizing modules of the rings from Corollary~\ref{thm:classgroup}.

\begin{thm} \label{tbyttrivial}
Assume that $Y$ is $t$-connected and satisfies Assumption (d).
  Then $\sfk_t(Y)$ has only trivial semidualizing modules.
\end{thm}

\begin{proof}
  The proof is similar to that of Corollary~\ref{thm:classgroup} by induction on $t$. The
  case $t=2$ is given by \cite[Theorem 3.10]{SWSeSpP1}. 
  For the induction step we augment the diagram from the proof of Corollary~\ref{thm:classgroup}.
$$\xymatrix{
\Cl(\sfk_t(Y))\ar[r]^-\cong
&\Cl(\sfk_t(Y)_x)\ar[r]^-\cong
&\Cl(\sfk_{t-1}(Z)[B_1]_{\mfx})
&\Cl(\sfk_{t-1}(Z))\ar[l]_-\cong
\\
\s_0(\sfk_t(Y))\ar[r]^-\rho\ar@{^(->}[u]
&\s_0(\sfk_t(Y)_x)\ar[r]^-\approx\ar@{^(->}[u]
&\s_0(\sfk_{t-1}(Z)[B_1]_{\mfx})\ar@{^(->}[u]
&\s_0(\sfk_{t-1}(Z))\ar[l]_-\approx\ar@{^(->}[u]
}$$
The vertical maps are the natural inclusions.
The second horizontal map on the bottom row is a bijection because it is induced by the ring isomorphism 
$\sfk_t(Y)_x\xra\cong\sfk_{t-1}(Z)[B_1]_\mfx$.
  The third horizontal map on the bottom row is a bijection by~\cite[Corollary~3.11(b)]{sather:divisor}.
A diagram chase shows that the map $\rho$ is injective. To show that it is surjective,
note that our induction hypothesis implies that $\sfk_{t-1}(Z)$ has only trivial semidualizing modules, hence
so does $\sfk_t(Y)_x$. It is straightforward to show that the two trivial semidualizing modules (free and dualizing) are in the image
of  $\rho$, that is, that $\rho$ is also surjective, as desired.
\end{proof}

\begin{defn} \label{type1.1}
Let $Y$ be $t$-connected (noting that Assumption~\ref{ass190807c} is still active). Recalling that an upper inside corner $T'_j=(c_j,d_j)$ is of type 1 if the $(t-1)$-minor based on $T'_j$
  contains at most one point in $B_1$, we say that 
  $T'_j=(c_j,d_j)$ is of \textbf{type~1.1} if the $(t-1)$-minor based on $T'_j$
  contains at least one point in $B_1$.
  \end{defn}

\begin{disc}\label{disc190808c}
With $Y$ as in the definition,
let $k^\bullet$ be the number of upper inside corners of $Y$ with type 1.1.
Conca~\cite[p.~458]{MR1413891} 
shows that
there are sub-ladders $Y_0,\ldots,Y_{k^\bullet}$ of $Y$ such that 
$\sfk_t(Y)=\Lambda^\sfk/(\ull)$ where $\Lambda^\sfk=\sfk_t(Y_0)\otimes_\sfk\cdots\otimes_\sfk\sfk_t(Y_{k^\bullet})$ and $\ull=\ell_1,\ldots,\ell_s$
is a sequence of linear 
forms 
that is $\Lambda^\sfk$-regular.
Moreover, each linear form $\ell_i$ is a difference of variables; the point is that the construction takes disjoint ladders and glues them together
by identifying certain entries. 
In particular, the $\ell_i$ are independent of the coefficient ring.
Furthermore, each sub-ladder $Y_i$ has no upper inside corners of type 1.1, i.e., $Y_i$ satisfies Assumption~(d).
In the case $t=2$, we have $Y=Y_0\#\cdots\# Y_{k^\bullet}$
as in~\cite[Notation~3.1]{SWSeSpP2}.

Corollary~\ref{cor190808a} implies that the map
$$\iota\colon\s_0(\sfk_t(Y_0))\times\cdots\times\s_0(\sfk_t(Y_{k^\bullet}))\to\s_0(\Lambda^\sfk)$$ 
given by
$([C_0],\ldots,[C_{k^\bullet}])\mapsto[C_0\otimes_\sfk\cdots\otimes_\sfk C_{k^\bullet}]$ is well-defined and bijective; and~\cite[Proposition~2.1(6)]{SWSeSpP2} implies that the base-change map $\s_0(\Lambda^\sfk)\to\s_0(\Lambda^\sfk/(\ull))$  is well-defined and injective.
One point of the next result is that the base-change map is bijective as well. 
Another point is that one can explicitly describe the image of this map, as follows.

Theorem~\ref{tbyttrivial} implies that each ring $\sfk_t(Y_i)$ has only the trivial semidualizing modules $\sfk_t(Y_i)$ and $\omega_{\sfk_t(Y_i)}$.
The natural map $\sfk_t(Y_i)\to \Lambda^\sfk$ is flat; in fact, free. 
We abuse notation and let $[\omega_{\sfk_t(Y_i)}]\in\Cl(\Lambda^\sfk)$
denote the image of $[\omega_{\sfk_t(Y_i)}]\in\Cl(\sfk_t(Y_i))$ under the induced map $\Cl(\sfk_t(Y_i))\to \Cl(\Lambda^\sfk)$, 
and we let $[\omega_{\sfk_t(Y_i)}]\in\Cl(\Lambda^\sfk/(\ull))$
denote the image of $[\omega_{\sfk_t(Y_i)}]\in\Cl(\Lambda^\sfk))$ under the induced map $\Cl(\Lambda^\sfk)\to\Cl(\Lambda^\sfk/(\ull))$. 
The preceding paragraph with the triviality of $\s_0(\sfk_t(Y_i))$
implies that every semidualizing $\Lambda^\sfk$-module is of the form $C\cong C_0\otimes_\sfk\cdots\otimes_\sfk C_{k^\bullet}$ where
for each $i$ we have $C_i\cong\sfk_t(Y_i)$ or $C_i\cong\omega_{\sfk_t(Y_i)}$. 
Set $\theta_i=0$ if $C_i\cong\sfk_t(Y_i)$ and $\theta_i=1$ otherwise.
Then in $\Cl(\Lambda^\sfk)$, we have 
$[C]=\sum_{i=1}^{k^\bullet}\theta_i[\omega_{\sfk_t(Y_i)}]$,
and the image of $[C]$ in $\Cl(\Lambda^\sfk/(\ull))=\Cl(\sfk_t(Y))$ is $\sum_{i=1}^{k^\bullet}\theta_i[\omega_{\sfk_t(Y_i)}]$ as well.
\end{disc}

A criterion for the Gorensteinness of $\sfk_t(Y_j)$ in the next result is in~\cite[Theorem~5.2]{MR1413891}.

\begin{thm} \label{sdmstbyt}
Assume that $Y$ is $t$-connected, and use the notation from Remark~\ref{disc190808c}.
Then
  \[
    \s_0(\sfk_t(Y))=\left\{\,\sum_{j=0}^{k^{\bullet}} \te_j [\omega_{\sfk_t(Y_j)}] \colon \te_j = 0 \text{ or } 1 \,\right\}\approx\prod_{j=0}^{k^{\bullet}}\s_0(\sfk_t(Y_j)).
  \]
  In particular, $|\s_0(\sfk_t(Y))| = |\s_0(\sfk_t(Y_0))| \cdots|\s_0(\sfk_t(Y_{k^{\bullet}}))|
  =2^{\ve_0 + \dots + \ve_{k^{\bullet}}}$, where $\ve_j=0$ if $\sfk_t(Y_j)$ is Gorenstein
  and $\ve_j=1$ otherwise.
\end{thm}

\begin{proof}
The proof is similar to that of Corollary~\ref{thm:classgroup} by induction on $t$.
In light of Remark~\ref{disc190808c}, it suffices to show that $|\s_0(\sfk_t(Y))| \leq |\s_0(\sfk_t(Y_0))| \cdots|\s_0(\sfk_t(Y_{k^{\bullet}}))|$.

The case $t=2$ follows from \cite[Theorem 3.12]{SWSeSpP2}.  If $t>2$, then we 
consider the diagram from the proof of Theorem~\ref{tbyttrivial}.
$$\xymatrix{
\Cl(\sfk_t(Y))\ar[r]^-\cong
&\Cl(\sfk_t(Y)_x)\ar[r]^-\cong
&\Cl(\sfk_{t-1}(Z)[B_1]_\mfx)
&\Cl(\sfk_{t-1}(Z))\ar[l]_-\cong
\\
\s_0(\sfk_t(Y))\ar@{^(->}[r]\ar@{^(->}[u]
&\s_0(\sfk_t(Y)_x)\ar[r]^-\approx\ar@{^(->}[u]
&\s_0(\sfk_{t-1}(Z)[B_1]_\mfx)\ar@{^(->}[u]
&\s_0(\sfk_{t-1}(Z))\ar[l]_-\approx\ar@{^(->}[u]
}$$
Because of the injectivity/bijectivity of the maps in the bottom row, 
it suffices to show that $|\s_0(\sfk_{t-1}(Z))| = |\s_0(\sfk_t(Y_0))| \cdots|\s_0(\sfk_t(Y_{k^{\bullet}}))|$.

To this end, the analysis of Remark~\ref{disc190808a} shows that 
there are $(t-1)$-connected subladders $Z_0,\ldots,Z_{k^\bullet}$ of $Z$ corresponding exactly to the subladders 
$Y_0,\ldots,Y_{k^\bullet}$ of $Y$ 
such that 
\begin{enumerate}[(1)]
\item \label{c190811a}
each $Z_\ell$ satisfies Assumption~(d),
\item  \label{c190811b}
$\sfk_{t-1}(Z_\ell)$ is Gorenstein if and only if $\sfk_t(Y_\ell)$ is Gorenstein,
and
\item  \label{c190811c}
there are integers $u$ and $0=k_0<\cdots<k_u<k_{u+1}=k^\bullet+1$ such that
the ladders $Z_{k_i}\cup\cdots\cup Z_{k_{i+1}-1}$ for $i=0,\ldots,u$ are the $(t-1)$-components of $Z$.
\end{enumerate}
Thus, our argument will be complete once we verify the next sequence of equalities.
\begin{align*}
|\s_0(\sfk_{t-1}(Z))| 
&= |\s_0(\sfk_{t-1}(Z_0\cup\cdots\cup Z_{k_1-1}))| \cdots|\s_0(\sfk_{t-1}(Z_{k_u}\cup\cdots\cup Z_{k^\bullet}))|
\\
&= |\s_0(\sfk_{t-1}(Z_0))| \cdots|\s_0(\sfk_{t-1}(Z_{k^{\bullet}}))|
\\
&= |\s_0(\sfk_t(Y_0))| \cdots|\s_0(\sfk_t(Y_{k^{\bullet}}))|
\end{align*}
The first equality is from Corollary~\ref{cor190808a}, and 
the second equality is from our induction hypothesis applied to the $(t-1)$-connected ladders $Z_{k_i}\cup\cdots\cup Z_{k_{i+1}-1}$.
The third equality is from Theorem~\ref{tbyttrivial}, with help from conditions~\eqref{c190811a}--\eqref{c190811b} above,
since each $Y_\ell$ satisfies Assumption~(d).
\end{proof}

\begin{ex}
We compute $\s_0(\sfk_3(L_i))$ for the ladders $L_i$ from Example~\ref{cornertypes}.  The sub-ladders $Y_{\ell}$ of each $L_i$, if they exist, are shown below; refer back to Example~\ref{cornertypes} for the original $L_i$ if necessary.

\begin{center}
    \begin{picture}(290,100)
      \put(60,85){$X_{13}$}
      \put(78,85){$X_{14}$}
      \put(96,85){$X_{15}$}
      \put(60,70){$X_{23}$}
      \put(78,70){$X_{24}$}
      \put(96,70){$X_{25}$}    
      \put(-30,55){$X_{31}$}
      \put(-12,55){$X_{32}$}
      \put(6,55){$X_{33}$}
      \put(24,55){$X_{34}$}
      \put(60,55){$X_{33}$}
      \put(78,55){$X_{34}$}
      \put(96,55){$X_{35}$}  
      \put(-30,40){$X_{41}$}
      \put(-12,40){$X_{42}$}
      \put(6,40){$X_{43}$}
      \put(24,40){$X_{44}$}
      \put(60,40){$X_{43}$}
      \put(78,40){$X_{44}$}
      \put(96,40){$X_{45}$} 
      \put(-30,25){$X_{51}$} 
      \put(-12,25){$X_{52}$}
      \put(6,25){$X_{53}$}
      \put(24,25){$X_{54}$}
      \put(-15, 10){\text{ Ladders }$Y_1, Y_2$ { for } $L_1$}
      \put(270,85){$X_{13}$}
      \put(288,85){$X_{14}$}
      \put(306,85){$X_{15}$}
      \put(270,70){$X_{23}$}
      \put(288,70){$X_{24}$}
      \put(306,70){$X_{25}$}
      \put(180,55){$X_{31}$}
      \put(198,55){$X_{32}$}
      \put(216,55){$X_{33}$}
      \put(234,55){$X_{34}$}
      \put(270,55){$X_{33}$}
      \put(288,55){$X_{34}$}
      \put(306,55){$X_{35}$}
      \put(180,40){$X_{41}$}
      \put(198,40){$X_{42}$}
      \put(216,40){$X_{43}$}
      \put(234,40){$X_{44}$}
      \put(180,25){$X_{51}$}
      \put(198,25){$X_{52}$}
      \put(216,25){$X_{53}$}
      \put(234,25){$X_{54}$}
       \put(200, 10){\text{ Ladders }$Y_1, Y_2$ { for } $L_2$}
 \end{picture}
  \end{center}

\begin{center}
    \begin{picture}(290,100)
      \put(20,85){$X_{13}$}
      \put(40,85){$X_{14}$}
      \put(60,85){$X_{15}$}
      \put(20,70){$X_{23}$}
      \put(40,70){$X_{24}$}
      \put(60,70){$X_{25}$}
      \put(-20,55){$X_{31}$}    
      \put(0,55){$X_{32}$}
      \put(20,55){$X_{33}$}
      \put(40,55){$X_{34}$}
      \put(60,55){$X_{35}$}  
      \put(-20,40){$X_{41}$}
      \put(0,40){$X_{42}$}
      \put(20,40){$X_{43}$}
      \put(40,40){$X_{44}$}
      \put(60,40){$X_{45}$}  
      \put(-20,25){$X_{51}$}
      \put(0,25){$X_{52}$}
      \put(20,25){$X_{53}$}
      \put(40,25){$X_{54}$}
      \put(60,25){$X_{55}$}
      \put(-20,10){$X_{61}$}
      \put(0,10){$X_{62}$}
      \put(20,10){$X_{63}$}
      \put(40,10){$X_{64}$}
       \put(-15, -5){\text{ Ladder }$L_3$: {no } $Y_{\ell}$ }
      \put(260,85){$X_{13}$}
      \put(280,85){$X_{14}$}
      \put(300,85){$X_{15}$}
      \put(260,70){$X_{23}$}
      \put(280,70){$X_{24}$}
      \put(300,70){$X_{25}$}
      \put(180,55){$X_{31}$}
      \put(200,55){$X_{32}$}
      \put(220,55){$X_{33}$}
      \put(260,55){$X_{33}$}
      \put(280,55){$X_{34}$}
      \put(300,55){$X_{35}$}
      \put(180,40){$X_{41}$}
      \put(200,40){$X_{42}$}
      \put(220,40){$X_{43}$}
      \put(180,25){$X_{51}$}
      \put(200,25){$X_{52}$}
      \put(220,25){$X_{53}$}
       \put(200,-5){\text{ Ladders } $Y_1, Y_2$ { for } $L_4$}
 \end{picture}
  \end{center}
\smallskip

In $Y = L_1$, the upper inside corner $T_1'$ has type 1.1, as per Definition~\ref{type1.1}, so $k^\bullet=1$.   For $\ell=1,2$, since $Y_\ell$ is rectangular and not square, the ring $\sfk_3(Y_\ell)$ is not Gorenstein by~\cite[Theorem~5.2]{MR1413891}. 
Thus, Theorem~\ref{sdmstbyt} implies that
$|\s_0(\sfk_3(L_1))|=4$.

For $Y = L_2$, decompose $Y$ similarly into $Y_1$ and $Y_2$ as displayed above.  In this case, each ladder $Y_\ell$ is rectangular, but only one is square. Thus, one of the rings $\sfk_3(Y_\ell)$ is  Gorenstein, while the other is not, again by~\cite[Theorem~5.2]{MR1413891}, thus, it follows that $|\s_0(\sfk_3(L_2))|=2$.

In $Y = L_3$, since $k^{\bullet}=0$, we do not decompose $L_3$.  In this case, since the underlying matrix of variables $X$ is not square, we see that 
$\sfk_3(Y)$ is not Gorenstein by~\cite[Theorem~5.2]{MR1413891}, hence, $|\s_0(\sfk_3(L_3))|=2$.
\smallskip

Lastly, $Y=L_4$ decomposes into two square subladders $Y_i$, thus each ring $\sfk_3(Y_\ell)$ is  Gorenstein by~\cite[Theorem~5.2]{MR1413891}, and hence,
$|\s_0(\sfk_3(L_4))|=1$. In particular, $\sfk_3(L_4)$ is Gorenstein.

\smallskip

\end{ex}

The next result is a version of Theorem~\ref{sdmstbyt} for $t$-disconnected ladders.

\begin{cor} \label{cor190812a}
Let $Y_1,\ldots,Y_s$ be the $t$-components of $Y$.
For $i=1,\ldots,s$, let $k_i^\bullet$ be the number of upper inside corners of $Y_i$ of type 1.1, and let
$Y_{i,0},\ldots,Y_{i,k_i^\bullet}$ be subladders of $Y_i$ as in Remark~\ref{disc190808c}.
Then
  \[
    \s_0(\sfk_t(Y))
    \approx\prod_{i=1}^s\prod_{j=0}^{k_i^{\bullet}}\s_0(\sfk_t(Y_{i,j})).
  \]
  In particular, $|\s_0(R)| = \prod_{i=1}^s\prod_{j=0}^{k_i^{\bullet}}|\s_0(\sfk_t(Y_{i,j}))|
  =2^{(\sum_{i=1}^s\sum_{j=0}^{k_i^{\bullet}}\ve_{i,j})}$, where $\ve_{i,j}=0$ if $\sfk_t(Y_{i,j})$ is Gorenstein
  and $\ve_{i,j}=1$ otherwise.
\end{cor}

\begin{proof}
Combine Corollary~\ref{cor190808a} and Theorem~\ref{sdmstbyt}.
\end{proof}

Now we return to our work toward removing the assumptions \eqref{thm181226a1} and \eqref{thm181226a2} 
from
Theorem~\ref{thm181226a}. 

\begin{disc}\label{disc190811a}
Assume that $Y$ is $t$-connected and use
the notation of Remark~\ref{disc190808c}.
Set $\Lambda^D=D_t(Y_0)\otimes_D\cdots\otimes_DD_t(Y_{k^\bullet})$.
Since the $\ell_i$'s from Remark~\ref{disc190808c} are differences of variables, it is straightforward to show that $\Lambda^D/(\ull)\cong D_t(Y)$.
Note that $\Lambda^D$ is a normal domain by Lemma~\ref{normaldomain} applied to the disjoint union of the $Y_i$.

We claim that the sequence $\ull$ is $\Lambda^D$-regular. 
Let $K_D=K^{\Lambda^D}(\ull)$ be the Koszul complex over $\Lambda^D$, and let $K_D^+$ denote the 
augmented Koszul complex
$$K_D^+: \qquad 0\to \Lambda^D\to\cdots\to \Lambda^D\to \Lambda^D/(\ull)\to 0.$$
As noted in Remark~\ref{disc190808c}, the sequence $\ull$ is $\Lambda^\sfk$-regular for every field $\sfk$,
hence $K_\sfk^+\cong K_\bbz^+\otimes_\bbz\sfk$ is exact.
Now, use~\cite[Lemma~2.6]{foxby:bcfm} as in the proof of Lemma~\ref{lem190329a} to conclude that $K_D^+$ is exact.
Since the sequence $\ull$ is homogeneous, the exactness of $K_D^+$ implies that the sequence is regular, as desired.
\end{disc}

The following lemmas allow Theorem~\ref{thm181226a} to be invoked in the proof of our main result, namely Theorem \ref{thm181230a}, which is a culmination of our efforts to describe the semidualizing modules of $A_t(Y)$, 
where $Y$ is not necessarily $t$-connected.

\begin{lem}\label{lem190103b}
When $Y$ is $t$-connected, the natural map $\s_0(D_t(Y))\stackrel{\phi}{\to}\s_0(L_t(Y))$ is bijective. 
\end{lem}

\begin{proof}
Lemma~\ref{lem190807b}  shows that $\phi$ is injective; it remains to show that 
$\phi$ is surjective. 
Use the notation from Remarks~\ref{disc190808c} and~\ref{disc190811a}.
Then Theorem \ref{sdmstbyt} shows that an arbitrary  $[C]\in\s_0(L_t(Y))$ is the image of 
$[C_0\otimes_L\cdots\otimes_LC_{k^\bullet}]$ under the natural map $\s_0(\Lambda^L)\to\s_0(\Lambda^L/(\underline\ell))$.
Moreover, \emph{loc.\ cit.} shows that for $i=0,\ldots,k^\bullet$ we have $C_i\cong L_t(Y_i)\cong L_t(Y_i)\otimes_{D_t(Y_i)}D_t(Y_i)$
or $C_i\cong \omega_{L_t(Y_i)}\cong L_t(Y_i)\otimes_{D_t(Y_i)}\omega_{D_t(Y_i)}$ since the map $D_t(Y)\to L_t(Y)$ is a localization map.
It follows that there are $[B_i]\in\s_0(D_t(Y_i))$ equal to $[D_t(Y_i)]$ or $[\omega_{D_t(Y_i)}]$ such that
$[C_0\otimes_L\cdots\otimes_LC_{k^\bullet}]$ is the image of $[B_0\otimes_L\cdots\otimes_LB_{k^\bullet}]\in\s_0(\Lambda^D)$ under the natural map
$\s_0(\Lambda^D)\to\s_0(\Lambda^L)$. 
Let $[B]\in\s_0(\Lambda^D/(\underline\ell))$ be the image of $[B_0\otimes_L\cdots\otimes_LB_{k^\bullet}]$ under the natural map
$\s_0(\Lambda^D)\to\s_0(\Lambda^D/(\underline\ell))$, as in the next diagram.
$$\xymatrix@C=8mm{
[B_0\otimes_L\cdots\otimes_LB_{k^\bullet}]\ar@{|->}[dd]\ar@{|->}[rrr]&&&[C_0\otimes_L\cdots\otimes_LC_{k^\bullet}]\ar@{|->}[dd]\\
&\s_0(\Lambda^D)\ar[r]\ar[d]&\s_0(\Lambda^L)\ar[d]&\\
[B]&\s_0(\Lambda^D/(\underline\ell))\ar[r]^-\phi&\s_0(\Lambda^D/(\underline\ell))&[C]
}$$
Commutativity of the diagram shows that $[C]=\phi([B])$, so $\phi$ is surjective.
\end{proof}

The version of Lemma~\ref{lem190103b} 
given next
considers ladders that are $t$-disconnected.

\begin{lem}\label{lem190812a}
The natural map $\phi\colon\s_0(D_t(Y))\to\s_0(L_t(Y))$ is bijective. 
\end{lem}

\begin{proof}
Let $Y_1,\ldots,Y_s$ be the $t$-components of $Y$.
Lemma~\ref{lem190103b} implies that each induced map $\phi_i\colon\s_0(D_t(Y_i))\to\s_0(L_t(Y_i))$ is bijective. 
Using the isomorphisms $D_t(Y)\cong D_t(Y_1)\otimes_D\cdots\otimes_D D_t(Y_s)$ and $L_t(Y)\cong L_t(Y_1)\otimes_L\cdots\otimes_L L_t(Y_s)$,
the horizontal maps in the following commutative diagram 
and the bijectivity of one of them are from Example~\ref{ex181224c} and Corollary~\ref{cor190808a}.
$$\xymatrix{
\prod_{i=1}^s\s_0(D_t(Y_i))\ar[r]\ar[d]_\approx^{\prod_i\phi_i}
&\s_0(D_t(Y))\ar@{^(->}[d]^\phi
\\
\prod_{i=1}^s\s_0(L_t(Y_i))\ar[r]^-\approx&\s_0(L_t(Y))}$$
The map $\phi$ is injective by Lemma~\ref{lem190807a}.
A straightforward diagram chase shows that $\phi$ is also surjective, as desired.
\end{proof}

The main theorem of this paper is below.
It is a version of Theorem~\ref{thm181226a} without the conditions \eqref{thm181226a1}--\eqref{thm181226a2}.
Note that the ladder $Y$ may be $t$-disconnected.
See Example~\ref{ex181224c} for a discussion of transitivity for $\tri$.

\begin{thm}
\label{thm181230a}
Let $A$ be a normal domain and $Y$ a ladder as in Assumption~\ref{ass190807c}. 
The natural maps below are bijections:
$$\s_0(A)\times\s_0(D_t(Y))\to\s_0(A_t(Y))\to\s_0(A_t(Y)_f)\times\s_0(K_t(Y)).$$ 
In particular,
with notation as in Corollary~\ref{cor190812a},
\begin{align*}
\s_0(A_t(Y))&\approx\s_0(A_t(Y)_f)\times\s_0(K_t(Y))\approx\s_0(A)\times\s_0(K_t(Y))\\
&\approx\s_0(A)\times\prod_{i=1}^s\prod_{j=0}^{k_i^{\bullet}}\s_0(\sfk_t(Y_{i,j}))\approx\s_0(A)\times\{0,1\}^e,
\end{align*}
where $e$ is the number of ladders $Y_{i,j}$ such that $K_t(Y_{i,j})$ is not Gorenstein.
Thus, $\s_0(A_t(Y))$ is finite if and only if $\s_0(A)$ is finite, and the relation $\tri$ on 
$\s_0(A_t(Y))$ is transitive if and only if the relation $\tri$ on $\s_0(A)$ is so.
\end{thm}

\begin{proof}
The bijections follow from Theorem~\ref{thm181226a}\eqref{thm181226a2}, Corollary~\ref{cor190812a},
and Lemma~\ref{lem190812a}.  Consequently, since 
$\{0,1\}^e$
is finite, it is clear that $\s_0(A_t(Y))$ and $\s_0(A)$ 
are either both finite or neither is.
Next, the set $\{0,1\}$ is totally ordered by $\tri$ (that is, $\leq$). Then, in particular, $\tri$ is a partial order on any finite product of copies of $\{0,1\}$; i.e., $\tri$ is transitive. Therefore, if $\tri$ is a partial order on $\s_0(A)$, then it is a partial order on $\s_0(A_t(Y))$, and vice versa.
\end{proof}

\begin{cor}
\label{cor190103a}
If $Y$ is $t$-connected and satisfies Assumption~(d),
then 
\begin{align*}
\s_0(A_t(Y))
\!\approx\!\s_0(A)\times\s_0(K_t(Y))
\approx\begin{cases}
\s_0(A)&\text{if $K_t(Y)$ is Gorenstein}\\
\s_0(A)\times\{0,1\}\!\!&\text{if $K_t(Y)$ is not Gorenstein.}
\end{cases}
\end{align*}
\end{cor}

\begin{proof}
This is the case  
of Theorem~\ref{thm181230a} with $s=1$ and $k_0^\bullet=0$.
\end{proof}

\begin{disc}\label{disc190813b}
It is natural to ask what $\s_0(\sfk[\![Y]\!]/I_t(Y))$ looks like, as is done for $\sfk[\![X]\!]/I_t(X)$ in~\cite{sather:divisor}.
In that work, the fact that $\sfk_t(X)$ satisfies Serre's condition $(R_2)$ is crucial, since it allows one to conclude that 
$\s_0(\sfk[\![X]\!]/I_t(X))\approx\s_0(\sfk_t(X))$; see~\cite[Corollary~3.13]{sather:divisor}.
If $\sfk_t(Y)$ were to satisfy $(R_2)$, then it would similarly hold that $\s_0(\sfk_t[\![Y]\!]/I_t(Y))\approx\s_0(\sfk_t(Y))$, and furthermore, 
it would simplify some of our regular sequence arguments above.
Thus, we pose the following question.
\end{disc}

\begin{question}\label{q190103b}
Under Assumption~\ref{ass190807c},
must $\sfk_t(Y)$ satisfy $(R_2)$?
\end{question}

\appendix
\section{Ladders and Ladder Determinantal Rings} \label{sec181225bx}

Throughout this section, $A$ will be a commutative ring with identity and $Y$ is a ladder, not necessarily path-connected or $t$-connected.
See \cite[pp.~169-170]{SWSeSpP1} and \cite{MR1413891} for the definitions of and notations for properties about ladders.

\

This section concerns general properties of ladder determinantal rings, and in particular, details some isomorphisms between such rings.  We will apply these results in the main body of the paper when we consider  their divisor class groups and semidualizing modules.  

We will recall the description of inside corners according to type. Example \ref{cornertypes} illustrates these various types of corners.

\begin{defn}[\protect{\cite[p.~467]{MR1413891}}] \label{defn:types}
An inside lower corner $S_i'=(a_i,b_i)$ of $Y$ is said to be of {\bf type 1} if the $(t-1)$-minor based on $(a_i,b_i)$ is contained in the ladder $Y$ and contains at most one point of the upper border $C$ of thickness 1; if so, then that point is $(a_i + t - 2,b_i + t - 2)$. If $S_i'$ is not of type 1, then it is of {\bf type 2}.  The definitions of the inside upper corner $T_j'=(c_j,d_j)$ types are analogous. Let $h^*$ be the number of $S_i'$ of type 1 and $k^*$ the number of $T_j'$ of type 1. The ladder $Y$ satisfies \textbf{Assumption~(d)} if for each 
$S_i'=(a_i,b_i)$ the $(t-1)$-minor based on $(a_i,b_i)$ is contained in the ladder $Y$ and contains no point of the upper border $C$.
\end{defn}

\begin{ex} \label{cornertypes} {\bf (Running Example)} Consider $\sfk_3(-)$ for each of the ladders below, and in particular, the 2-minor based at each inside corner $S_1'=(3,3)$.  In the ladders in the left column, (3,3) is a type 1 corner: the 2-minor based on $X_{33}$ is contained in $L_i$ and contains either one point of the border $C$ ($X_{44}$ in $L_1$, indicated with a dotted box), or no points of $C$ ($L_3$); for these examples, we have $h^*=1=k^*$.  The corner (3,3) in each of the ladders in the right column is of type 2.  In $L_2$, the 2-minor based on $X_{33}$ is contained in the ladder, but contains more than one point of $C$ ($X_{34}$ and $X_{44}$), while in $L_4$, the 2-minor  based on $X_{33}$ is not contained in the ladder; for these examples, we have $h^*=0=k^*$.

\begin{center}
    \begin{picture}(290,100)
      \put(17,35){\line(0,1){29}}
      \put(17,64){\line(1,0){42}}
      \put(17,35){\line(1,0){21}}
      \put(59,50){\line(0,1){14}}
      \multiput(38,35)(0,1){15}{\line(0,1){.25}}
      \multiput(38,50)(1,0){21}{\line(1,0){.25}}
      \multiput(38,35)(1,0){22}{\line(1,0){.25}}
      \multiput(59,35)(0,1){15}{\line(0,1){.25}}
      \put(20,85){$X_{13}$}
      \put(40,85){$X_{14}$}
      \put(60,85){$X_{15}$}
      \put(20,70){$X_{23}$}
      \put(40,70){$X_{24}$}
      \put(60,70){$X_{25}$}    
      \put(-20,55){$X_{31}$}
      \put(0,55){$X_{32}$}
      \put(20,55){$X_{33}$}
      \put(40,55){$X_{34}$}
      \put(60,55){$X_{35}$}  
      \put(-20,40){$X_{41}$}
      \put(0,40){$X_{42}$}
      \put(20,40){$X_{43}$}
      \put(40,40){$X_{44}$}
      \put(60,40){$X_{45}$} 
      \put(-20,25){$X_{51}$} 
      \put(0,25){$X_{52}$}
      \put(20,25){$X_{53}$}
      \put(40,25){$X_{54}$}
      \put(3, 10){\text{ Ladder }$L_1$}
      \put(217,35){\line(0,1){29}}
      \put(217,64){\line(1,0){21}}
      \put(217,35){\line(1,0){21}}
      \multiput(259,35)(0,1){30}{\line(0,1){.25}}
      \multiput(238,35)(0,1){30}{\line(0,1){.25}}
      \multiput(238,64)(1,0){22}{\line(0,1){.25}}
      \multiput(238,35)(1,0){22}{\line(0,1){.25}}
      \put(220,85){$X_{13}$}
      \put(240,85){$X_{14}$}
      \put(260,85){$X_{15}$}
      \put(220,70){$X_{23}$}
      \put(240,70){$X_{24}$}
      \put(260,70){$X_{25}$}
      \put(180,55){$X_{31}$}
      \put(200,55){$X_{32}$}
      \put(220,55){$X_{33}$}
      \put(240,55){$X_{34}$}
      \put(260,55){$X_{35}$}
      \put(180,40){$X_{41}$}
      \put(200,40){$X_{42}$}
      \put(220,40){$X_{43}$}
      \put(240,40){$X_{44}$}
      \put(180,25){$X_{51}$}
      \put(200,25){$X_{52}$}
      \put(220,25){$X_{53}$}
      \put(240,25){$X_{54}$}
       \put(200, 10){\text{ Ladder }$L_2$}
 \end{picture}
  \end{center}

\begin{center}
    \begin{picture}(290,100)
      \put(17,35){\line(0,1){29}}
      \put(17,64){\line(1,0){42}}
      \put(17,35){\line(1,0){42}}
      \put(59,35){\line(0,1){29}}
      \put(20,85){$X_{13}$}
      \put(40,85){$X_{14}$}
      \put(60,85){$X_{15}$}
      \put(20,70){$X_{23}$}
      \put(40,70){$X_{24}$}
      \put(60,70){$X_{25}$}
      \put(-20,55){$X_{31}$}    
      \put(0,55){$X_{32}$}
      \put(20,55){$X_{33}$}
      \put(40,55){$X_{34}$}
      \put(60,55){$X_{35}$}  
      \put(-20,40){$X_{41}$}
      \put(0,40){$X_{42}$}
      \put(20,40){$X_{43}$}
      \put(40,40){$X_{44}$}
      \put(60,40){$X_{45}$}  
      \put(-20,25){$X_{51}$}
      \put(0,25){$X_{52}$}
      \put(20,25){$X_{53}$}
      \put(40,25){$X_{54}$}
      \put(60,25){$X_{55}$}
      \put(-20,10){$X_{61}$}
      \put(0,10){$X_{62}$}
      \put(20,10){$X_{63}$}
      \put(40,10){$X_{64}$}
       \put(3, -5){\text{ Ladder }$L_3$}
      \put(220,85){$X_{13}$}
      \put(240,85){$X_{14}$}
      \put(260,85){$X_{15}$}
      \put(220,70){$X_{23}$}
      \put(240,70){$X_{24}$}
      \put(260,70){$X_{25}$}
      \put(180,55){$X_{31}$}
      \put(200,55){$X_{32}$}
      \put(220,55){$X_{33}$}
      \put(240,55){$X_{34}$}
      \put(260,55){$X_{35}$}
      \put(180,40){$X_{41}$}
      \put(200,40){$X_{42}$}
      \put(220,40){$X_{43}$}
      \put(180,25){$X_{51}$}
      \put(200,25){$X_{52}$}
      \put(220,25){$X_{53}$}
       \put(200,0){\text{ Ladder }$L_4$}
 \end{picture}
  \end{center}
\end{ex}
\smallskip

The next result allows us to work with ladders that are not $t$-connected.
Essentially, the sub-ladders $Y_i$ are $t$-connected components of $Y$.

\begin{lem}
\label{lem190402a}
There are sub-ladders $Y_1,\ldots,Y_s$ of $Y$ such that
\begin{enumerate}[\rm(1)]
\item\label{lem190402a1}
each ladder $Y_i$ is non-empty and $t$-connected;
\item\label{lem190402a2}
$Y$ is the disjoint union  $Y_1\cup\cdots\cup Y_s$; 
and
\item\label{lem190402a3}
every $t$-minor of $Y$ is contained in one of the $Y_i$.
\end{enumerate}
\end{lem}

\begin{proof}
Induct on $|Y|$. 
If $Y$ is $t$-connected, then we are done with $s=1$.
Thus, assume that $Y$ is $t$-disconnected.  This implies that there are non-empty sub-ladders $Y'$ and $Y''$ of $Y$ such that
$Y$ is the disjoint union  $Y'\cup Y''$ and
every $t$-minor of $Y$ is contained in $Y'$ or $Y''$. Now apply the induction hypothesis to $Y'$ and $Y''$.
\end{proof}

\begin{rmk} 
The subladders $Y_i$ in Lemma~\ref{lem190402a} are unique, as one can show, so we refer to them
as the {\bf $\pmb{t}$-components} of the ladder $Y$.
\end{rmk}

Because of the next result, below we are able to use Fact~\ref{lem190103a} to find the semidualizing modules for $A_t(Y)$ when $A$ is a normal domain.

\begin{lem} \label{normaldomain} 
If $A$ is a normal domain, then so is $A_t(Y)$.
\end{lem}

\begin{proof}
Let $Y_1,\ldots,Y_s$ be the $t$-components of $Y$. We induct on $s$.  When $s=1$, $Y$ is $t$-connected. By \cite[Proposition 3.3]{Co}, $R = \mathsf k_t(Y)$ is a normal domain. Apply \cite[Proposition 3.12]{bruns:dr}: since $S=\bbz_t(Y)$ is faithfully flat over $\mathbb Z$ and $R \cong S \otimes_{\mathbb Z} \sfk$ is a normal domain, $A_t(Y)$ is a normal domain since $A$ is a normal domain.

For the inductive step, note that if $Y'=Y\ssm Y_1$, then $A_t(Y)\cong A_t(Y_1)\otimes_A A_t(Y')\cong [A_t(Y_1)]_t(Y')$. 
By the base case, the ring $A_t(Y_1)$ is a normal domain, hence it follows, by induction, that the same is true for $A_t(Y)\cong [A_t(Y_1)]_t(Y')$.
\end{proof}

The result below is a basic tool for the proof of Lemma~\ref{lem190329b}. It can be stated in much more generality than stated here, but this 
is the
version that we need. 

\begin{lem}\label{lem190329a}
Let $\phi\colon L\to M$ be a homomorphism between torsion-free abelian groups, i.e., between flat $\bbz$-modules.
The following conditions are equivalent:
\begin{enumerate}[\rm(i)]
\item \label{lem190329a1}
$\phi$ is an isomorphism;
\item \label{lem190329a2'}
for each ring $A$, the map $\phi\otimes_{\bbz}A$ is an isomorphism; 
\item \label{lem190329a2}
for each field $\sfk$, the map $\phi\otimes_{\bbz}\sfk$ is an isomorphism; and
\item \label{lem190329a3}
the map $\phi\otimes_{\bbz}\bbq$ is an isomorphism and, for each prime number $p$, the map $\phi\otimes_{\bbz}(\bbz/p\bbz)$ is an isomorphism.
\end{enumerate}
\end{lem}

\begin{proof}
The implications \eqref{lem190329a1}$\implies$\eqref{lem190329a2'}$\implies$\eqref{lem190329a2}$\implies$\eqref{lem190329a3} are routine.

\eqref{lem190329a3}$\implies$\eqref{lem190329a1}
Assume that the map $\phi\otimes_{\bbz}\bbq$ is an isomorphism and, for each prime number $p$, the map $\phi\otimes_{\bbz}(\bbz/p\bbz)$ is an isomorphism.
In other words, for each prime ideal $\p\subseteq\bbz$, the map $\phi\otimes_{\bbz}\kappa(\p)$ is an isomorphism where $\kappa(\p)$ is the field of fractions of $\bbz/\p$.
Consider the bounded chain complex of flat $\bbz$-modules
$$X=\cone(\phi)=(0\to L\xra\phi M\to 0)$$
which satisfies $X\otimes_{\bbz}\kappa(\p)\cong\cone(\phi\otimes_{\bbz}\kappa(\p))$.
The fact that $\phi\otimes_{\bbz}\kappa(\p)$ is an isomorphism for all $\p$ implies that $X\otimes_{\bbz}\kappa(\p)$ is exact for all $\p$. 
Since $X$ is a bounded complex of flat $\bbz$-modules, it follows from~\cite[Lemma~2.6]{foxby:bcfm} that $X$ is exact, that is, that $\phi$ is an
isomorphism.
\end{proof}

The following notation, included for convenience, was introduced in~\cite[p.~462]{MR1413891}.

\begin{notn}\label{notn190329a}
For $i=1,\ldots,h+1$ let $F_i\in A[Y]$ denote the $(t-1)$-minor based on the outside lower corner $S_i$ of $Y$, and let $f_i\in A_t(Y)$ be the residue of $F_i$. 
Set $F=F_1\cdots F_{h+1}$ and $f=f_1\cdots f_{h+1}$. 
For $j=1,\ldots,k+1$ let $G_j\in A[Y]$ denote the $(t-1)$-minor based on the outside upper corner $T_j$ of $Y$, and let $g_j\in A_t(Y)$ be the residue of $G_j$. 
Set $G=G_1\cdots G_{k+1}$ and $g=g_1\cdots g_{k+1}$. 
\end{notn}

Much of our work is based on careful localization. We describe our tools for this next, 
beginning with a version of~\cite[Lemma~4.1]{MR1413891} with fewer restrictions on $A$.

\begin{lem}\label{lem190329b}
Assume that $Y$ is $t$-connected.
Let $B$ be the set of points of the lower border with thickness $(t-1)$ of $Y$, and
let $C$ be the set of points of the upper border with thickness $(t-1)$ of $Y$.
Set $Y_1=\{P\in Y\mid P\preccurlyeq(a_1,b_1+t-2)\}$ and  
 $Y_2=\{P\in Y\mid P\preccurlyeq(c_1-t+2,d_1)\}$, and set $B_1=B\ssm Y_1$ and $C_1=C\ssm Y_2$.
Then one has isomorphisms of localizations
\begin{align*}
A_t(Y)_{f_1}
&\cong A_t(Y_1)[B_1]_{f_{11}}
&
A_t(Y)_{g_1}
&\cong A_t(Y_2)[C_1]_{g_{11}}
&
A_t(Y)_f
&\cong A[B]_F
\end{align*}
where $f_{11}$ is the residue of $F_1$ in $A_t(Y_1)[B_1]$
and $g_{11}$ is the residue of $G_1$ in $A_t(Y_2)[C_1]$.
\end{lem}

\begin{proof}
We verify the first isomorphism;
the others are verified similarly. 

The case where $A$ is a field is covered in~\cite[Lemma~4.1]{MR1413891}.
Moreover, the proof of~\cite[Lemma~4.1]{MR1413891} provides a natural map
\begin{align*}
A_t(Y_1)[B_1]_{f_{11}}
&\xra{\psi_A} A_t(Y)_{f_1}
\end{align*}
such that $\psi_{\sfk}$ is an isomorphism for each field $\sfk$.
Consider the map
$$\bbz_t(Y_1)[B_1]_{f_{11}}
\xra{\psi_\bbz} \bbz_t(Y)_{f_1}
$$
between flat $\bbz$-modules. It is straightforward to show that $\psi_{\sfk}=\psi_{\bbz}\otimes_{\bbz}\sfk$;
since this map is an isomorphism for each $\sfk$, Lemma~\ref{lem190329a} implies that 
$\psi_{A}=\psi_{\bbz}\otimes_{\bbz}A$ is an isomorphism for all $A$.
\end{proof}

Next, we remove the $t$-connected hypothesis for the isomorphism we use below.

\begin{lem}\label{lem190329bxxx}
Let $B$ be the set of points of the lower border with thickness $(t-1)$ of $Y$.
Then one has
$A_t(Y)_f
\cong A[B]_F$.
\end{lem}

\begin{proof}
Let $Y_1,\ldots,Y_s$ be the $t$-components of $Y$. 
We induct on $s$, where the base case $s=1$ follows from Lemma~\ref{lem190329b}.
Re-order the $Y_i$ if necessary to assume without loss of generality that $X_{1,n}\in Y_1$, and set $\wti Y=Y\ssm Y_1$.
Set $$f_{(1)}=\prod_{X_{S_j}\in Y_1}f_j\qquad \text{and} \qquad\wti f=f/f_{(1)}=\prod_{X_{S_j}\in Y'}f_j$$
Also, set $B_1=B\cap Y_1$ and $\wti B=B\ssm B_1$. Since $Y_1$ is part of a $t$-disconnection of $Y$,
it follows that $B_1$ is the set of points of the lower border with thickness $(t-1)$ of $Y_1$
and $\wti B$ is the set of points of the lower border with thickness $(t-1)$ of $\wti Y$.
Thus, our base case and inductive assumption explain the fourth isomorphism in the following display.
\begin{align*}
A_t(Y)_f
&\cong (A_t(Y_1))_t(\wti Y)_{f_{(1)}\wti f}
\cong [(A_t(Y_1))_t(\wti Y)_{f_{(1)}}]_{\wti f}
\cong (A_t(Y_1)_{f_{(1)}})_t(\wti Y)_{\wti f}\\
&\cong (A[B_1]_{F_{(1)}})[\wti B]_{\wti F}
\cong [(A[B_1])[\wti B]_{F_{(1)}}]_{\wti F}
\cong A[B_1\cup \wti B]_{F_{(1)}\wti F}
\cong A[B]_F
\end{align*}
The first isomorphism is from the fact that $Y_1$ forms part of a $t$-disconnection of $Y$, along with the definitions above.
The remaining steps are straightforward.
\end{proof}

\begin{assumption}\label{ass190715a}
Assume 
that $Y$ is $t$-connected and $t>2$.
Let $B_1$ be the set of points of $Y$ of the lower border with thickness 1, and set
$Z=Y \setminus B_1$.
Also, set $x=x_{S_1}x_{S_2}\cdots x_{S_{h+1}}\in \sfk_t(Y)$
and $\mfx=X_{S_1}X_{S_2}\cdots X_{S_{h+1}}\in \sfk_t[Y]$.
(We use $\mfx$ here instead of $X$ since $X$ is the $m\times n$ matrix of variables containing our ladder $Y$.)
See Remark~\ref{disc190808a} for a discussion of the connectedness properties of $Z$ and its corners.
\end{assumption}

Our next goal is to prove Proposition~\ref{prop:invertx} which, for our $t$-connected ladder $Y$, gives an isomorphism
$$
\sfk_t(Y)_x \cong \sfk_{t-1}(Z)[B_1]_{\mfx}.$$
We  prove this result by a series of lemmas.
The proof outline is similar to that of \cite[Proposition 4.1 (2)]{Co}, except that
in the ladder $Z$, we do not relabel the variables after we delete $B_1$ from $Y$.
In pursuit of the isomorphism above, we define endomorphisms $\psi,\chi$ on $\sfk[Y]_{\mfx}$ as follows.

\begin{defn} \label{defn:psichi}
Continue with Assumption~\ref{ass190715a}.
   Let $X_{ij} \in Y$, for a fixed pair $(i,j)$,
  and let $U=U(i,j)$ be the interval $\{w \in \bbn \mid i>a_{w-1} \text{ and }
  j>b_w\}$. Define $\psi\colon\sfk[Y]_{\mfx}\to \sfk[Y]_{\mfx}$ by
  \begin{align*}
    \psi(X_{ij}) &= X_{ij} + \sum_{\substack{r\geq 1 \\ \{u_1 < u_2 < \dots < u_r\} \subseteq U}}
    \frac{X_{a_{u_1-1},j}X_{a_{u_2-1},b_{u_1}}\cdots X_{a_{u_r-1},b_{u_{r-1}}}
    X_{i,b_{u_r}}}{X_{a_{u_1-1},b_{u_1}}X_{a_{u_2-1},b_{u_2}}\cdots
    X_{a_{u_r-1},b_{u_r}}}\\
    &= X_{ij} + \sum_{\substack{r\geq 1 \\ \{u_1 < u_2 < \dots < u_r\} \subseteq U}}
    \frac{X_{a_{u_1-1},j}X_{a_{u_2-1},b_{u_1}}\cdots X_{a_{u_r-1},b_{u_{r-1}}}
    X_{i,b_{u_r}}}{X_{S_{u_1}}X_{S_{u_2}}\cdots X_{S_{u_r}}}.
  \end{align*}
In particular, we have $\psi(X_{ij})=X_{ij}$ for all $X_{ij} \in B_1$ since $U(i,j) = \varnothing$ for these variables.  Similarly, we define $\chi\colon \sfk[Y]_{\mfx}\to\sfk[Y]_{\mfx}$ by
  \[
    \chi(X_{ij}) = X_{ij} + \sum_{\substack{r\geq 1 \\ \{u_1 < u_2 < \dots < u_r\} \subseteq U}}
    (-1)^r \frac{X_{a_{u_1-1},j}X_{a_{u_2-1},b_{u_1}}\cdots
    X_{a_{u_r-1},b_{u_{r-1}}}X_{i,b_{u_r}}}
    {X_{S_{u_1}}X_{S_{u_2}}\cdots X_{S_{u_r}}}.
  \]
 \end{defn}

To provide some clarity, we include an example.

\begin{ex} In ladder $L_3$ of Example \ref{cornertypes}, if $(i,j) = (2,4)$, then $U(2,4) = \{1\}$ since $a_0 = 1$ and $b_1 = 3$; hence $\psi(X_{24}) = X_{24} + \frac{X_{14}X_{23}}{X_{13}}$. For $(i,j) = (4,5)$, we have $U(4,5) = \{1,2\}$, hence we have the distinct chains 
$\{1\}, \{2\}, \{1 < 2\}$ in $U$.  Thus, $$\psi(X_{45}) = X_{45} + \frac{X_{15}X_{43}}{X_{13}} + \frac{X_{35}X_{41}}{X_{31}} + \frac{X_{15}X_{33}X_{41}}{X_{13}X_{31}}.$$
\end{ex}

\begin{lem}\label{lem190807a}
The maps $\psi,\chi$ from Definition~\ref{defn:psichi} are inverses of each other.
\end{lem}

\begin{proof}
  We  show that $\chi(\psi(X_{ij}))=X_{ij}$. Similar arguments  show that $\psi(\chi
  (X_{ij}))=X_{ij}$. Note that $\chi(\psi(X_{ij}))=X_{ij}$ for all $X_{ij} \in B_1$ since for these variables, $U(i,j) = \varnothing$.
  For $X_{ij} \notin B_1$, we consider the expression
  \begin{equation} \label{eqn:chi}
    \chi \left(
    \frac{X_{a_{u_1-1},j}X_{a_{u_2-1},b_{u_1}}\cdots X_{a_{u_r-1},b_{u_{r-1}}}
    X_{i,b_{u_r}}}{X_{S_{u_1}}X_{S_{u_2}}\cdots X_{S_{u_r}}} \right) .
  \end{equation}
  We have 
  \begin{align*}
    &\chi(X_{a_{u_1-1},j}) = X_{a_{u_1-1},j}\\ 
    &\hspace{10mm}+
    \sum_{u(i,j) \leq z_1 < z_2 < \dots < z_\ell \leq u_1-1} (-1)^\ell
    \frac{X_{a_{z_1-1},j}X_{a_{z_2-1},b_{z_1}}\cdots X_{a_{z_\ell-1},b_{z_{\ell-1}}}
    X_{a_{u_1-1},b_{z_\ell}}}{X_{S_{z_1}}X_{S_{z_2}}\cdots X_{S_{z_\ell}}}\\
    &\chi(X_{a_{u_2-1},b_{u_1}}) = X_{a_{u_2-1,b_{u_1}}}\\ 
    &\hspace{8mm}+
    \sum_{u_1+1 \leq z_1 < z_2 < \dots < z_\ell \leq u_2-1} (-1)^\ell
    \frac{X_{a_{z_1-1},b_{u_1}}X_{a_{z_2-1},b_{z_1}}\cdots X_{a_{z_\ell-1},b_{z_{\ell-1}}}
    X_{a_{u_2-1},b_{z_\ell}}}{X_{S_{z_1}}X_{S_{z_2}}\cdots X_{S_{z_\ell}}}\\
    & \hspace{45pt} \vdots \\
    &\chi(X_{i,b_{u_r}}) = X_{i,b_{u_r}}\\ 
    &\hspace{13mm}+
    \sum_{u_r+1 \leq z_1 < z_2 < \dots < z_\ell \leq v(i,j)} (-1)^\ell
    \frac{X_{a_{z_1-1},b_{u_r}}X_{a_{z_2-1},b_{z_1}}\cdots X_{a_{z_\ell-1},b_{z_{\ell-1}}}
    X_{i,b_{z_\ell}}}{X_{S_{z_1}}X_{S_{z_2}}\cdots X_{S_{z_\ell}}}\,.
  \end{align*}
  Therefore, when we expand expression \eqref{eqn:chi}, all the terms are of the form
  \begin{equation} \label{eqn:expanded}
    (-1)^{r'-r} \frac{X_{a_{u_1-1},j}X_{a_{u_2-1},b_{u_1}}\cdots
    X_{a_{u_{r'}-1},b_{u_{r'-1}}}X_{i,b_{u_{r'}}}}
    {X_{S_{u_1}}X_{S_{u_2}}\cdots X_{S_{u_{r'}}}},
  \end{equation}
  where $u(i,j) \leq u_1 < u_2 < \dots < u_{r'} \leq v(i,j)$ and $r' \geq r \geq 1$.
  Note that all variables that appear in \eqref{eqn:expanded}
  have index $\lneqq (i,j)$, hence the coefficient of $X_{ij}$ in $\chi
  (\psi(X_{ij}))$ is 1. Now in \eqref{eqn:expanded}, we fix $r'$ and vary $r$.
  Then the coefficient of
  \[
    \frac{X_{a_{u_1-1},j}X_{a_{u_2-1},b_{u_1}}\cdots
    X_{a_{u_{r'}-1},b_{u_{r'-1}}}X_{i,b_{u_{r'}}}}
    {X_{S_{u_1}}X_{S_{u_2}}\cdots X_{S_{u_{r'}}}}
  \]
  is $\sum_{r=0}^{r'} \binom{r'}{r} (-1)^{r'-r}=0$, where the coefficient
  $(-1)^{r'}$ appears in the definition of $\chi(X_{ij})$.
  Therefore, $\chi(\psi(X_{ij}))=X_{ij}$.
\end{proof}

\begin{rmk} \label{rmk:lindep} 
  With the notation and assumptions of Definition~\ref{defn:psichi}, it holds that
  \[
    \psi(X_{ij}) = X_{ij} + \sum_{w \in U} \frac{X_{i,b_w}}{X_{S_w}} \psi(X_{a_{w-1,j}}).
  \]
\end{rmk}

The next few lemmas are used to show that the maps $\psi$ and $\chi$ respect certain ideals of minors.
In them, we use the notation $|M|=[p_1, \dots, p_t \mid q_1,\dots,q_t]$ for the $t$-minor 
of the matrix $M$ involving the variables located at the points $(p_i,q_j)$.

\begin{lem} \label{lem:outsidecorner}
With assumptions as in~\ref{ass190715a},
  let $(a_{i-1},b_i)$ be an outside lower corner of $Y$ with $1 \leq i \leq h+1$.  If there are integers $\nu_j$ for $j=1,\ldots,t-1$ such that $b_i<\nu_1<\nu_2<\dots<\nu_{t-1}$ and the $t$-minor $$|M| =[a_{i-1},a_{i-1}+1,a_{i-1}+2,\dots,a_{i-1}+t-1 \mid b_i,\nu_1,\nu_2,\dots,\nu_{t-1}] \in I_t(Y)$$ then
  \begin{equation} \label{eqn:outsidecorner}
      \psi(|M|) = X_{S_i}([a_{i-1}+1,a_{i-1}+2,\dots,a_{i-1}+t-1 \mid \nu_1,\nu_2,\dots,\nu_{t-1}]+E)
  \end{equation}
  where $E$ is a linear combination of
  $(t-1)$-minors of the form $$[a_{i-1}+1,a_{i-1}+2,\dots,a_{i-1}+t-1 \mid  \gs_1,\gs_2,\dots,\gs_{t-1}]$$
  and $\gs_j=\nu_j$ or $\gs_j =b_w$ with $w \in U(a_{i-1},\nu_j)$, and not all $\gs_j=\nu_j$.
  In particular, $b_i<\gs_j \leq \nu_j$.
\end{lem}

\begin{proof}
  Apply $\psi$ to all the entries in the minor
  $|M|$.  Prior to calculating the determinant of $M$ we perform elementary row operations, using Remark~\ref{rmk:lindep} and the fact that the $X_{S_w}$ are units, so that the first column is reduced to
  $[ \begin{matrix}
      X_{S_i} & 0 & \cdots & 0
    \end{matrix} ]^T$.
  The $(j+1)$st column becomes
$$      \begin{bmatrix}
      X_{a_{i-1},\nu_j} \\ X_{a_{i-1}+1,\nu_j} \\ \vdots \\ X_{a_{i-1}+t-1,\nu_j}
    \end{bmatrix} 
    + \sum_{w \in U(a_{i-1},\nu_j)} \frac{\psi(X_{a_{w-1},\nu_j})}{X_{S_w}}
    \begin{bmatrix}
      X_{a_{i-1},b_w} \\ X_{a_{i-1}+1,b_w} \\ \vdots \\ X_{a_{i-1}+t-1,b_w}
    \end{bmatrix}.
$$
  The Lemma then follows by expanding the determinant along the first column.
\end{proof}

\begin{lem} \label{lem:outsidecornergen}
Continue with Assumption~\ref{ass190715a}.
Let $(a_{i_0-1},b_{i_0})$ be an outside lower corner, $1 \leq i_0 \leq h+1$. Let $\mu_j, \nu_j \in \mathbb N$ be such that 
  $a_{i_0-1}<\mu_1<\mu_2<\dots<\mu_{t-1}$ and $b_{i_0}<\nu_1<\nu_2<\dots<\nu_{t-1}$,
  and set 
  $$|M|:= [a_{i_0-1},\mu_1,\mu_2,\dots,\mu_{t-1} \mid b_{i_0},\nu_1,\nu_2,\dots,\nu_{t-1}] \in I_t(Y).$$
  Then one has
  \begin{equation} \label{eqn:outsidecornergen}
    \psi(|M|)= X_{S_{i_0}}([\mu_1,\mu_2,\dots,\mu_{t-1} \mid \nu_1,\nu_2,\dots,\nu_{t-1}]+E)
  \end{equation}
  where $E$ is a linear combination of
  $(t-1)$-minors of the form $$[\rho_1,\rho_2,\dots,\rho_{t-1} \mid \gs_1,\gs_2,\dots,\gs_{t-1}]$$
  where $\rho_i \in \{ a_{w-1} \mid w \in U(\mu_i,b_{i_0})\} \cup \{\mu_i\}$,
  $\gs_j \in \{ b_w \mid w \in U(a_{i_0-1},\nu_j) \} \cup \{ \nu_j \}$, and it does not happen that
  all $\rho_i=\mu_i$ and $\gs_j=\nu_j$ at the same time. In particular, one has
  $a_{i_0-1} < \rho_i \leq \mu_i$ and $b_{i_0} < \gs_j \leq \nu_j$.
\end{lem}

\begin{proof}
  Let $W = \{a_{w-1} \mid w \in U(\mu_{t-1},b_{i_0})\} \ssm \{\mu_1,\dots,\mu_{t-1}\}$.
  If $W = \varnothing$, then the proof is similar to that of Lemma~\ref{lem:outsidecorner}.
  Otherwise, let $u=\min W$ and $v=\max W$.
  Form the minor
  \begin{center}
    $|M'| =  \left|
    \begin{tabular}{ c | c c c c c }
                         & $X_{a_u,b_{i_0}}$ & $X_{a_u,\nu_1}$ & $X_{a_u,\nu_2}$ & $\cdots$ & $X_{a_u,\nu_{t-1}}$\\
      $I_{|W|}$ & $\vdots$               &                               &                  &                               & $\vdots$\\
                         & $X_{a_v,b_{i_0}}$ & $X_{a_v,\nu_1}$ & $X_{a_v,\nu_2}$ & $\cdots$ & $X_{a_v,\nu_{t-1}}$\\
      \hline
      \pmb{0}                 &&& \rule{0pt}{15pt} $M$
    \end{tabular}
    \right|$,
  \end{center}
  so that $|M'|=|M|$. Apply $\psi$ to all the entries in $M'$, and row reduce as in
  Lemma~\ref{lem:outsidecorner}, so that $\psi(|M'|)=\psi(|M|)$ becomes
   \begin{center}
    $\left|
    \begin{tabular}{ c c c | c l l l l }
      1 &           & 0  & 0 & $X_{a_u,\nu_1}+\cdots$ & $X_{a_u,\nu_2}+\cdots$ & $\cdots$ & $X_{a_u,\nu_{t-1}}+\cdots$\\
      & $\ddots$ & & $\vdots$               &                               &                  &                               & \hspace{15pt} $\vdots$\\
      * &           &  1  & 0 & $X_{a_v,\nu_1}+\cdots$ & $X_{a_v,\nu_2}+\cdots$ & $\cdots$ & $X_{a_v,\nu_{t-1}}+\cdots$\\
      \hline
      \rule{0pt}{12pt} 0  & $\cdots$ & 0 & $X_{S_{i_0}}$ & $X_{a_{i_0-1},\nu_1}+\cdots$ & $X_{a_{i_0-1},\nu_2}+\cdots$ & $\cdots$ & $X_{a_{i_0-1},\nu_{t-1}}+\cdots$\\
                                                              &&& 0 & $X_{a_{\mu_1},\nu_1}+\cdots$ & $X_{a_{\mu_1},\nu_2}+\cdots$ & $\cdots$ & $X_{a_{\mu_1},\nu_{t-1}}+\cdots$\\
           &   *    &                                          & $\vdots$ &&&& \hspace{15pt} $\vdots$\\
                                                              &&& 0 & $X_{a_{\mu_{t-1}},\nu_1}+\cdots$ & $X_{a_{\mu_{t-1}},\nu_2}+\cdots$ & $\cdots$ & $X_{a_{\mu_{t-1}},\nu_{t-1}}+\cdots$
    \end{tabular}
    \right|$.
  \end{center}
  Expand the above determinant along the $(|W|+1)$st column,
  and then along the first $|W|$ columns. The conditions for $\gs_j$ then follow.
  Now we note that when we row reduce $\psi(|M'|)$, we use the row that
  contains $\psi(X_{a_w,b_{i_0}})$ to row reduce the row that contains
  $\psi(X_{\mu_i,b_{i_0}})$ if and only if $a_w<\mu_i$. So when 
  the row-reduced determinant is expanded along the first $|W|+1$ columns,
  the remaining rows either have index $\mu_1,\dots,\mu_{t-1}$,
  or some of the $\mu_i$ can be replaced with smaller indices among
  $\{a_{w-1} \mid w \in U(\mu_{t-1},b_{i_0})\}$. Therefore $a_{i_0-1}<\rho_i \leq \mu_i$.
\end{proof}

\begin{cor} \label{cor:contains}
 Continue with Assumption ~\ref{ass190715a}.
  Let $a_{i_0-1}<\mu_1<\mu_2<\dots<\mu_{t-1}$ and $b_{i_0}<\nu_1<\nu_2<\dots<\nu_{t-1}$
  be such that $1 \leq i_0 \leq h+1$ and
 $$[a_{i_0-1},\mu_1,\mu_2,\dots,\mu_{t-1} \mid b_{i_0},\nu_1,\nu_2,\dots,\nu_{t-1}] \in I_t(Y).$$
  Let $|N|=[\mu_1,\mu_2,\dots,\mu_{t-1} \mid \nu_1,\nu_2,\dots,\nu_{t-1}]$. Then
  \begin{equation} \label{eqn:chiofn}
    \chi(|N|)= X^{-1}_{S_{i_0}}([a_{i_0-1},\mu_1,\mu_2,\dots,\mu_{t-1} \mid
    b_{i_0},\nu_1,\nu_2,\dots,\nu_{t-1}]+E)
  \end{equation}
  where $E$ is a linear combination of
  $t$-minors of the form $$[a_{i_0-1},\rho_1,\rho_2,\dots,\rho_{t-1} \mid b_{i_0},\gs_1,\gs_2,\dots,\gs_{t-1}]$$
  with $a_{i_0-1} < \rho_i \leq \mu_i$ and $b_{i_0} < \gs_j \leq \nu_j$.
  In particular, we have $\psi(\langle I_t(Y)\rangle) \supseteq \langle I_{t-1}(Z)\rangle$ in $\sfk[Y]_{\mfx}$.
\end{cor}

\begin{proof}
  Let $\gz = (\gz_1,\dots,\gz_{t-1})$ and $\xi = (\xi_1,\dots,\xi_{t-1})$, where
  $\gz_i=\mu_i-(a_{i_0-1}+i)$ and $\xi_j=\nu_j-(b_{i_0}+j)$ for $1 \leq i,j \leq t-1$.
  Note that these sequences are non-decreasing.
  The proof is by induction on $(\xi,\gz)$, where the values of $\gz$ and $\xi$ each form a finite subset of
  $\bbn^{t-1}$. The ordering on $(\xi,\gz)$ is given by applying the reverse lexicographic order to $\xi$,
  to $\gz$, and then to $(\xi,\gz)$. Consider the base case when $\xi=\gz=0$.
  Then \eqref{eqn:outsidecorner} gives
  \begin{gather*}
    \psi([a_{i_0-1},\dots,a_{i_0-1}+t-1 \mid b_{i_0},\dots, b_{i_0}+t-1])
    = X_{S_{i_0}} |N|, \text{ so}\\
    \chi(|N|) = X^{-1}_{S_{i_0}} [a_{i_0-1},\dots,a_{i_0-1}+t-1 \mid b_{i_0},\dots, b_{i_0}+t-1].
  \end{gather*}

  In the induction step, let us consider $(\xi,\gz) \neq (0,0)$. The term $E$ in equation \eqref{eqn:outsidecornergen}
  is a  linear combination of 
  $(t-1)$-minors of the form $N'=[\rho_1,\rho_2,\dots,\rho_{t-1} \mid \gs_1,\gs_2,\dots,\gs_{t-1}]$
  where $a_{i_0-1} < \rho_i \leq \mu_i$, $b_{i_0} < \gs_j \leq \nu_j$, and $\rho_i<\mu_i$ for at least one $i$
  or $\gs_j<\nu_j$ for at least one $j$. Then one can apply $\chi$
  to both sides of \eqref{eqn:outsidecornergen}, rearrange, and apply the induction hypothesis
  to the terms $N'$ in $E$.

  Finally, we let $|N|=[\mu_1,\mu_2,\dots,\mu_{t-1} \mid \nu_1,\nu_2,\dots,\nu_{t-1}]$
  be a generator of $I_{t-1}(Z)$. Then $X_{\mu_1,\nu_1} \notin B_1$.
  Let $i_0 \in U(\mu_1,\nu_1) \neq \varnothing$. Then \eqref{eqn:chiofn} shows that
  $|N| \in \psi(I_t(Y))$. Therefore $\psi(I_t(Y)) \supseteq I_{t-1}(Z)$.
\end{proof}

\begin{lem} \label{lem:notoutsidecorner}
In addition to Assumption~\ref{ass190715a},
 suppose that $Y$ does not have an outside corner at $(\mu_0,\nu_0)$.
  Let $\mu_0<\mu_1<\mu_2<\dots<\mu_{t-1}$ and $\nu_0<\nu_1<\nu_2<\dots<\nu_{t-1}$
  be such that $|M|:=[\mu_0,\mu_1,\mu_2,\dots,\mu_{t-1} \mid \nu_0,\nu_1,\nu_2,\dots,\nu_{t-1}] \in I_t(Y)$.
  Then 
  \[
    \psi(|M|) = [\mu_0,\mu_1,\mu_2,\dots,\mu_{t-1} \mid \nu_0,\nu_1,\nu_2,\dots,\nu_{t-1}]+E
  \]
  where $E$ is a linear combination of $t$-minors of the form
  $$[\rho_0,\rho_1,\rho_2,\dots,\rho_{t-1} \mid \gs_0,\gs_1,\gs_2,\dots,\gs_{t-1}]$$
  where 
  \begin{align*}\rho_i &\in \{ a_{w-1} \mid w \in U(\mu_i,\nu_0) \ssm U(\mu_0,\nu_0) \} \cup \{\mu_i\}\\
  \gs_j &\in \{ b_w \mid w \in U(\mu_0,\nu_j) \} \cup \{ \nu_j \}
  \end{align*} 
  and it does not happen that 
  all $\rho_i=\mu_i$ and $\gs_j=\nu_j$ at the same time. In particular, 
  $\mu_0 \leq \rho_i \leq \mu_i$ and $\gs_j \leq \nu_j$.
\end{lem}

\begin{proof}
  The proof is similar to that of Lemma~\ref{lem:outsidecornergen} with the following modifications.
  Let $W = \{a_{w-1} \mid w \in U(\mu_{t-1},\nu_0) \ssm U(\mu_0,\nu_0) \}
  \ssm \{\mu_0,\mu_1,\dots,\mu_{t-1}\}$. When row reducing $\psi(|M'|)$,
     if
  $\mu_0 \in \{ a_{w-1} \mid w \in U(\mu_{t-1},\nu_0) \}$,
 use only the row that contains the entry $\psi(X_{\mu_0,\nu_0})$.
  Finally, after row reduction, expand along the first $|W|$ columns.
\end{proof}

\begin{notn}\label{notn190701a}
Continue with Assumption~\ref{ass190715a}, and let $Z=Y\ssm B_1$.
We extend the definitions of $\q_i,\q'_i,\p_j$ in \cite[pp.~459, 463]{MR1413891}.
For each lower outside corner $S_i$ as in \cite[p.~457]{MR1413891}, set
$$\mathfrak Q_i\!=\!
\begin{cases}
\{X_{pq}\!\in\! Y\mid a_{i-1}\leq p\leq a_{i-1}+t-2\}&\hspace{-7pt}\text{if $i=1$ or $S_{i-1}'$ has type 1}\\
\{X_{pq}\!\in\! Y\mid \text{$a_{i-1}\leq p\leq a_{i-1}+t-2$ and $q\leq d_j$}\}&\hspace{-7pt}\text{if $i> 1$ and $S_{i-1}'$ has type 2}
\end{cases}
$$
where, in the second case, $T_j'$ is the companion corner for $S_{i-1}'$. A slogan for this is
$$\mathfrak Q_i=
\begin{cases}
\{\text{rows of $Y$ involved in $F_i$}\}&\text{if $i=1$ or $S_{i-1}'$ has type 1}\\
\{\text{partial rows of $Y$ involved in $F_i$}\}&\text{if $i\neq 1$ and $S_{i-1}'$ has type 2.}
\end{cases}
$$
Set
\begin{align*}
Q_i(Y)&=I_{t-1}(X_{pq}\in A[Y]\mid X_{pq}\in\mathfrak Q_i) + I_t(Y) \subseteq A[Y]\\
\q_i(Y)&=Q_i(Y)/I_t(Y) \subseteq A_t(Y)
\end{align*}
and note that $\q_i$ is a height-1 prime ideal of $A_t(Y)$ containing $f_i$;
this is verified as in~\cite{MR1413891}.
The ideal $\q_i'$ is defined similarly using columns associated to $S_i$: 
$$\mathfrak Q_i'=
\begin{cases}
\{X_{pq}\in Y\mid b_{i}\leq q\leq b_{i}+t-2\}&\text{if $i=h+1$ or $S_{i}'$ has type 1}\\
\{X_{pq}\in Y\mid \text{$b_{i}\leq q\leq b_{i}+t-2$ and $p\leq c_j$}\}& \text{if $i\leq h$ and $S_{i}'$ has type 2}
\end{cases}
$$
where, in the second case, $T_j'$ is the companion corner for $S_i'$. A slogan for this is 
$$\mathfrak Q_i'=
\begin{cases}
\{\text{columns of $Y$ involved in $F_i$}\}&\text{if $i=h+1$ or $S_{i}'$ has type 1}\\
\{\text{partial columns of $Y$ involved in $F_i$}\}&\text{if $i\leq h$ and $S_{i}'$ has type 2.}
\end{cases}
$$
Set 
\begin{align*}
Q_i'(Y)&=I_{t-1}(X_{pq}\in A[Y]\mid X_{pq}\in\mathfrak Q_i') + I_t(Y)\subseteq A[Y]\\
\q_i'(Y)&=Q'_i(Y)/I_t(Y)\subseteq A_t(Y)
\end{align*}
noting that $\q_i'$ is another height-1 prime ideal of $A_t(Y)$ containing $f_i$.
For each inside corner $T_j'$, set
\begin{align*}
\mathfrak P_j&=
\begin{cases}
\{X_{pq}\in Y\mid (p,q) \leq T'_j\} & \text{if $T'_j$ has type 1}\\
\varnothing & \text{if $T'_j$ has type 2}
\end{cases}\\
P_j(Y)&=I_{t-1}(X_{pq}\in A[Y]\mid X_{pq}\in\mathfrak P_j) + I_t(Y)\subseteq A[Y]\\
\p_j(Y)&=P_j(Y)/I_t(Y)\subseteq A_t(Y)
\end{align*}
and note that $\p_j$ is a height-1 prime in $A_t(Y)$ whenever $\p_j\neq (0)$.

The ideals $\q_i(Z),\q'_i(Z),\p_j(Z)$ are defined similarly, but we number these
ideals as we do in $Y$, even if $Z$ is not $(t-1)$-connected.
\end{notn}

\begin{disc}\label{disc190808a}
It is straightforward to build examples where $Z=Y\ssm B_1$ is $(t-1)$-disconnected, even though $Y$ is $t$-connected; in particular, this is illustrated below with $t = 3$ for the ladders $L_2$ and $L_4$ from Example \ref{cornertypes}, and the $Z$ highlighted by the boxes.
The ladders $L_1$ and $L_3$ are examples where $Z$ is $(t-1)$-connected.
In what follows, we wish to consider the ideals $\q_i,\p_j$ associated to the rings $\sfk_3(L_n)$. 

\begin{center}
    \begin{picture}(290,100)
      \put(38,80){\line(1,0){41}}
      \put(38,50){\line(0,1){30}}
      \put(-2,50){\line(1,0){40}}
      \put(-2,21){\line(0,1){28}}
      \put(-2,21){\line(1,0){60}}
      \put(58,21){\line(0,1){15}}
      \put(58,36){\line(1,0){21}}
      \put(79,36){\line(0,1){44}}
      \put(20,85){$X_{13}$}
      \put(40,85){$X_{14}$}
      \put(60,85){$X_{15}$}
      \put(20,70){$X_{23}$}
      \put(40,70){$X_{24}$}
      \put(60,70){$X_{25}$}    
      \put(-20,55){$X_{31}$}
      \put(0,55){$X_{32}$}
      \put(20,55){$X_{33}$}
      \put(40,55){$X_{34}$}
      \put(60,55){$X_{35}$}  
      \put(-20,40){$X_{41}$}
      \put(0,40){$X_{42}$}
      \put(20,40){$X_{43}$}
      \put(40,40){$X_{44}$}
      \put(60,40){$X_{45}$} 
      \put(-20,25){$X_{51}$} 
      \put(0,25){$X_{52}$}
      \put(20,25){$X_{53}$}
      \put(40,25){$X_{54}$}
      \put(3, 10){\text{ Ladder }$L_1$ }
      \put(238,80){\line(1,0){40}}
      \put(238,50){\line(0,1){30}}
      \put(198,50){\line(1,0){61}}
      \put(198,20){\line(0,1){30}}
      \put(198,20){\line(1,0){61}}
      \put(259,20){\line(0,1){30}}
      \put(259,50){\line(1,0){19}}
      \put(278,50){\line(0,1){30}}
      \put(220,85){$X_{13}$}
      \put(240,85){$X_{14}$}
      \put(260,85){$X_{15}$}
      \put(220,70){$X_{23}$}
      \put(240,70){$X_{24}$}
      \put(260,70){$X_{25}$}
      \put(180,55){$X_{31}$}
      \put(200,55){$X_{32}$}
      \put(220,55){$X_{33}$}
      \put(240,55){$X_{34}$}
      \put(260,55){$X_{35}$}
      \put(180,40){$X_{41}$}
      \put(200,40){$X_{42}$}
      \put(220,40){$X_{43}$}
      \put(240,40){$X_{44}$}
      \put(180,25){$X_{51}$}
      \put(200,25){$X_{52}$}
      \put(220,25){$X_{53}$}
      \put(240,25){$X_{54}$}
       \put(200, 10){\text{ Ladder }$L_2$}
 \end{picture}
  \end{center}

\begin{center}
    \begin{picture}(290,100)
      \put(38,80){\line(1,0){41}}
      \put(38,50){\line(0,1){30}}
      \put(-2,50){\line(1,0){40}}
      \put(-2,6){\line(0,1){44}}
      \put(-2,6){\line(1,0){60}}
      \put(58,6){\line(0,1){15}}
      \put(58,21){\line(1,0){21}}
      \put(79,21){\line(0,1){59}} 
      \put(20,85){$X_{13}$}
      \put(40,85){$X_{14}$}
      \put(60,85){$X_{15}$}
      \put(20,70){$X_{23}$}
      \put(40,70){$X_{24}$}
      \put(60,70){$X_{25}$}
      \put(-20,55){$X_{31}$}    
      \put(0,55){$X_{32}$}
      \put(20,55){$X_{33}$}
      \put(40,55){$X_{34}$}
      \put(60,55){$X_{35}$}  
      \put(-20,40){$X_{41}$}
      \put(0,40){$X_{42}$}
      \put(20,40){$X_{43}$}
      \put(40,40){$X_{44}$}
      \put(60,40){$X_{45}$}  
      \put(-20,25){$X_{51}$}
      \put(0,25){$X_{52}$}
      \put(20,25){$X_{53}$}
      \put(40,25){$X_{54}$}
      \put(60,25){$X_{55}$}
      \put(-20,10){$X_{61}$}
      \put(0,10){$X_{62}$}
      \put(20,10){$X_{63}$}
      \put(40,10){$X_{64}$}
       \put(3, -5){\text{ Ladder }$L_3$}
      \put(239,80){\line(1,0){39}}
      \put(239,50){\line(0,1){30}}
      \put(239,50){\line(1,0){39}}
      \put(198,50){\line(1,0){39}}
      \put(198,20){\line(0,1){30}}
      \put(198,20){\line(1,0){38}}
      \put(237,20){\line(0,1){30}}
      \put(278,50){\line(0,1){30}}
      \put(220,85){$X_{13}$}
      \put(240,85){$X_{14}$}
      \put(260,85){$X_{15}$}
      \put(220,70){$X_{23}$}
      \put(240,70){$X_{24}$}
      \put(260,70){$X_{25}$}
      \put(180,55){$X_{31}$}
      \put(200,55){$X_{32}$}
      \put(220,55){$X_{33}$}
      \put(240,55){$X_{34}$}
      \put(260,55){$X_{35}$}
      \put(180,40){$X_{41}$}
      \put(200,40){$X_{42}$}
      \put(220,40){$X_{43}$}
      \put(180,25){$X_{51}$}
      \put(200,25){$X_{52}$}
      \put(220,25){$X_{53}$}
       \put(200,0){\text{ Ladder }$L_4$}
 \end{picture}
  \end{center}
\medskip

Continuing with our running example \ref{cornertypes}, the $\q_i$'s are displayed in Table~\ref{tableq} below. The $\q_i'$'s are similar, using columns.
\begin{table}[h]
\begin{center}
\begin{tabular}{|ll|ll|}
\hline
$Y=L_1:$ & $\q_1=I_2\begin{pmatrix}x_{13}&x_{14}&x_{15}\\x_{23}&x_{24}&x_{25}\end{pmatrix}$ & $Z:$ & $\q_1=(x_{24},x_{25})$ \\
& $\q_2=I_2\begin{pmatrix}x_{31}&x_{32}&x_{33}&x_{34}&x_{35}\\x_{41}&x_{42}&x_{43}&x_{44}&x_{45}\end{pmatrix}$ & & $\q_2=(x_{42},x_{43},x_{44},x_{45})$ \\
\hline
$Y=L_2:$ & $\q_1=I_2\begin{pmatrix}x_{13}&x_{14}&x_{15}\\x_{23}&x_{24}&x_{25}\end{pmatrix}$ & $Z:$ & $\q_1=(x_{24},x_{25})$ \\
& $\q_2=I_2\begin{pmatrix}x_{31}&x_{32}&x_{33}&x_{34}\\x_{41}&x_{42}&x_{43}&x_{44}\end{pmatrix}$ & & $\q_2=(x_{42},x_{43},x_{44})$ \\
\hline
$Y=L_3:$ & $\q_1=I_2\begin{pmatrix}x_{13}&x_{14}&x_{15}\\x_{23}&x_{24}&x_{25}\end{pmatrix}$ & $Z:$ & $\q_1=(x_{24},x_{25})$ \\
& $\q_2=I_2\begin{pmatrix}x_{31}&x_{32}&x_{33}&x_{34}&x_{35}\\x_{41}&x_{42}&x_{43}&x_{44}&x_{45}\end{pmatrix}$ & & $\q_2=(x_{42},x_{43},x_{44},x_{45})$ \\
\hline
$Y=L_4:$ & $\q_1=I_2\begin{pmatrix}x_{13}&x_{14}&x_{15}\\x_{23}&x_{24}&x_{25}\end{pmatrix}$ & $Z:$ & $\q_1=(x_{24},x_{25})$ \\
& $\q_2=I_2\begin{pmatrix}x_{31}&x_{32}&x_{33}\\x_{41}&x_{42}&x_{43}\end{pmatrix}$ & & $\q_2=(x_{42},x_{43})$ \\
\hline
\end{tabular}
\end{center}
\caption{} \label{tableq}
\end{table}

The upper inside corners of $Z$ are all upper inside corners of $Y$, but not vice versa in general.
Thus, we analyze the cases here.

If $T_j'$ is an upper inside corner of $Y$ with type 1, then $T_j'$ is also an upper inside corner of $Z$ with type 1, 
and $T_j'$ does not cause a $(t-1)$-disconnection of $Z$. 
(See ladders $L_1$ and $L_3$ above.)
In this case, the definitions of $\q_i,\q'_i,\p_j$ coincide with those in \cite{MR1413891}.
In particular, if $Y$ satisfies Assumption~(d), then $Z$ is $(t-1)$-connected and satisfies Assumption~(d) with respect to $t-1$.

If $T_j'$  is an upper inside corner of $Y$ with type 2, then $T_j'$ may or may not be an upper inside corner of $Z$, and 
$T_j'$ may or may not cause a $(t-1)$-disconnection of $Z$. 
(See ladders $L_2$ and $L_4$ above.)
In this case, let $S_i'$ be the companion corner for $T_j'$.
If $\max\{c_j-a_i,d_j-b_i\}<t-2$, then $I_{t-1}(X_{pq} \in Y \mid (p,q) \leq T'_j\,)=(0)$, and similarly for $Z$ if $T'_j \in Z$.
If $T'_j \notin Z$, then $\p_j(Z)$ is, in a sense, meaningless.
(This is the case for ladder $L_4$ above.)
If $d_j-b_i=t-2$, then $I_{t-1}(X_{pq} \in Y \mid (p,q) \leq T'_j\,)+I_t(Y)=\q'_i$ (as in ladder $L_2$ above).
If $c_j-a_i=t-2$, then $I_{t-1}(X_{pq} \in Y \mid (p,q) \leq T'_j\,)+I_t(Y)=\q_{i+1}$.
In all cases, $T'_j$ does not contribute new ideals to $\Cl(A_t(Y))$ nor $\Cl(A_{t-1}(Z))$,
hence, we may as well define $\p_j(Y)=(0)$ and $\p_j(Z)=(0)$.

The table below shows the ideals $\p_j$, which is only $\p_1$, for the running example(s).

\begin{table}[h]
\begin{center}
\begin{tabular}{|ll|ll|}
\hline
$Y=L_1:$ & $\p_1=I_2\begin{pmatrix}&&x_{13}&x_{14}\\&&x_{23}&x_{24}\\x_{31}&x_{32}&x_{33}&x_{34}\\x_{41}&x_{42}&x_{43}&x_{44}\end{pmatrix}$ & $Z:$ & $\p_1=I_1\begin{pmatrix}&&x_{24}\\&&x_{34}\\x_{42}&x_{43}&x_{44}\end{pmatrix}$ \\
\hline
\rule{0pt}{12pt}$Y=L_2:$ & $\p_1=(0)$ & $Z:$ & $\p_1=(0)$\\
\hline
$Y=L_3:$ & $\p_1=I_2\begin{pmatrix}&&x_{13}&x_{14}\\&&x_{23}&x_{24}\\x_{31}&x_{32}&x_{33}&x_{34}\\x_{41}&x_{42}&x_{43}&x_{44}\\x_{51}&x_{52}&x_{53}&x_{54}\end{pmatrix}$ & $Z:$ & $\p_1=I_1\begin{pmatrix}&&x_{24}\\&&x_{34}\\x_{42}&x_{43}&x_{44}\\x_{52}&x_{53}&x_{54}\end{pmatrix}$ \\
\hline
\rule{0pt}{12pt}$Y=L_4:$ & $\p_1=(0)$ & $Z:$ & $\p_1=(0)$\\
\hline
\end{tabular}
\end{center}
\caption{} \label{tablep}
\end{table}
\end{disc}

\begin{cor} \label{cor:correspondence}
Continue with Assumption~\ref{ass190715a}.
 Then $\psi(I_t(Y)) = I_{t-1}(Z)$, $\psi(Q_i(Y))=Q_i(Z)$ and $\psi(P_j(Y))=P_j(Z)$ in $\sfk[Y]_{\mfx}$.
\end{cor}

\begin{proof}
  We first show that $\psi(I_t(Y)) = I_{t-1}(Z)$.
  The inclusion $\supseteq$ is given by Corollary~\ref{cor:contains}, hence it is only necessary to
  establish $\subseteq$.

  Let $\mu_0<\mu_1<\mu_2<\dots<\mu_{t-1}$ and $\nu_0<\nu_1<\nu_2<\dots<\nu_{t-1}$
  be such that $|M|:=[\mu_0,\mu_1,\mu_2,\dots,\mu_{t-1} \mid \nu_0,\nu_1,\nu_2,\dots,\nu_{t-1}]$
  is a generator of $I_t(Y)$. We need to show that the terms of $\psi(|M|)$ from   
  Lemma~\ref{lem:notoutsidecorner} are in $I_{t-1}(Z)$.  By Lemma~\ref{lem:outsidecornergen}, it is sufficient to only consider the case when $(\mu_0,\nu_0)$ is not an outside corner of $Y$.
  Consequently,
  $X_{\mu_0,\nu_0+1} \notin B_1$ or $X_{\mu_0+1,\nu_0} \notin B_1$.
  If $X_{\mu_0,\nu_0+1} \notin B_1$, then we expand all determinants along the
  first column to get $\psi(|M|) \in I_{t-1}(Z)$.
  Next, if $X_{\mu_0,\nu_0} \notin B_1$, then $X_{\mu_0,\nu_0+1} \notin B_1$, so we
  may assume that $X_{\mu_0+1,\nu_0} \notin B_1$, but $X_{\mu_0,\nu_0} \in B_1$.
  By Lemma~\ref{lem:notoutsidecorner}, we have $\gs_j \in \{ b_w \mid
  w \in U(\mu_0,\nu_j)\} \cup \{ \nu_j\}$, so $\nu_0 \leq \gs_j \leq \nu_j$.
  Then we expand all determinants in $\psi(|M|)$ along the first row to get
  $\psi(|M|) \in I_{t-1}(Z)$.

  The proofs of $\psi(Q_i(Y))=Q_i(Z)$ and $\psi( P_j(Y))=P_j(Z)$ 
  are similar
  (by considering $t-1$ instead of $t$).
\end{proof}

\begin{prop} \label{prop:invertx}
  The maps $\psi$ and $\chi$ from Definition~\ref{defn:psichi} induce isomorphisms
  \[
    \sfk_t(Y)_{x} \cong \sfk_{t-1}(Z)[B_1]_{\mfx}.
  \]
\end{prop}

\begin{proof}
  From Corollary~\ref{cor:correspondence} it follows that $\psi(I_t(Y)) =  I_{t-1}(Z)$ in $\sfk[Y]_{\mfx}$,
  so the isomorphisms $\psi$ and $\chi$ from Lemma~\ref{lem190807a} induce the given isomorphism.
\end{proof}

\providecommand{\bysame}{\leavevmode\hbox to3em{\hrulefill}\thinspace}
\providecommand{\MR}{\relax\ifhmode\unskip\space\fi MR }
\providecommand{\MRhref}[2]{%
  \href{http://www.ams.org/mathscinet-getitem?mr=#1}{#2}
}
\providecommand{\href}[2]{#2}

\end{document}